\documentclass[12pt]{amsart}
\usepackage[latin9]{inputenc}
\usepackage{hyperref}
\usepackage{color}
\usepackage{verbatim}
\usepackage{amsthm}
\usepackage{amssymb}
\usepackage{multicol}
\usepackage{tabularx}
\usepackage{enumitem}
\usepackage{graphicx}

\usepackage{pgf,tikz}
\usetikzlibrary{arrows}

\makeatletter

\newtheorem{case}{Case}
\newtheorem{subcase}{Subcase}[case]
\newtheorem{subsubcase}{Subcase}[subcase]
\newtheorem{subsubsubcase}{Subcase}[subsubcase]

\def\qed{\hfill
\ifhmode\unskip\nobreak\fi\quad\ifmmode\Box\else$\Box$\fi\\ }
\newtheorem{theorem}{Theorem}
\newtheorem{cor}[theorem]{Corollary}
\newtheorem{lemma}[theorem]{Lemma}

\newtheorem{claim}{Claim}

\numberwithin{equation}{section}
\numberwithin{figure}{section}
\usepackage{enumitem}		% customizable list environments
\usepackage{fullpage}
\newtheorem{thm}[theorem]{Theorem}
\newtheorem*{thm*}{Theorem}
\newtheorem*{conj*}{Conjecture}
\newtheorem{conj}[]{Conjecture}
\usepackage{tikz}
\usepackage{pgf}
\newcount\mycount

\usepackage{caption}
\usepackage[labelformat=simple,labelfont={}]{subcaption}

\usetikzlibrary{shapes}
\usepackage{graphicx}
\makeatother

    \providecommand{\questionname}{Question}

\newcommand{\sU}{\mathcal{U}}
\begin{document}

%{\hfill{\small \bf\today }}

%{\hfill{\small \currfilename }}

\title{Strong Chromatic Index of Subcubic Planar Multigraphs}

\author{A.V. Kostochka}
\thanks{Department of Mathematics, University of Illinois, Urbana, IL, 61801,
USA.  This author's research is supported in part
by NSF grant  DMS-1266016 and by grants 12-01-00631 and 12-01-00448  of
the Russian Foundation for Basic Research.
E-mail address:  \texttt{kostochk@math.uiuc.edu}}

\author{X. Li}
\thanks{ Department of Mathematics, Huazhong Normal University, Wuhan, 430079,
China.
This author's research is supported in part by the Natural Science Foundation of China
(11171129)  and by Doctoral Fund of Ministry of Education of China
(20130144110001).
E-mail address: \texttt{xwli68@mail.ccnu.edu.cn}}

\author{W. Ruksasakchai}
\thanks{Department of Mathematics, Faculty of Science, KhonKaen University, KhonKaen, 40002, Thailand.  This author's research is supported in part by Development and Promotion of Science and Technology Talents Project (DPST).  E-mail address:  \texttt{watcharintorn1@hotmail.com}}

\author{M. Santana}
\thanks{Department of Mathematics, University of Illinois, Urbana, IL, 61801,
USA. This author's research is supported in part by the NSF grants DMS-1266016 ``AGEP-GRS'' and DMS 08-38434 ``EMSW21 - MCTP: Research Experience for Graduate Students.''  
E-mail address: \texttt{santana@illinois.edu.}}

\author{T. Wang}
\thanks{Institute of Applied Mathematics, Henan University, Kaifeng, 475004, P. R. China; College of Mathematics and Information Science, Henan University, Kaifeng, 475004, P. R. China;  This author's research is supported by the National Natural Science Foundation of China (11101125) and partially supported by the Fundamental Research Funds for Universities in Henan.  This research was done while the author was visiting the University of Illinois at Urbana-Champaign.  The author would like to thank Prof. Kostochka for his hospitality.
E-mail address: \texttt{wangtao@henu.edu.cn}}

\author{G. Yu}
\thanks{ Department of Mathematics, The College of William and Mary, Williamsburg,
VA, 23185, USA; Department of Mathematics, Huazhong Normal University,
Wuhan, 430079, China.
This author's research is supported in part by NSA grant
H98230-12-1-0226.
 E-mail address:  \texttt{gyu@wm.edu}}

\begin{abstract}
The strong chromatic index of a multigraph is the minimum $k$ such that the edge set can be $k$-colored requiring that each color class induces a matching.  We verify a conjecture of Faudree, Gy\'{a}rf\'{a}s, Schelp and Tuza, showing that every planar multigraph with maximum degree at most 3 has strong chromatic index at most 9, which is sharp.

This paper is to appear in European J. Combin. 51 (2016) 380--397.
\end{abstract}

\maketitle
\noindent{\small{Mathematics Subject Classification: 05C15 (05C10)}}{\small \par}

\noindent{\small{Keywords: subcubic graphs, strong edge-coloring, strong chromatic index, planar graphs.}}{\small \par}

{\begin{flushright}
Dedicated to the memory of Ralph J. Faudree.
\end{flushright}}

%%%%%%%%%%%%%%%%%%%%%%%%%%%%%%%%%%%%%%%%%%%

\section{Introduction}

All multigraphs in this paper are loopless.  A \emph{strong $k$-edge-coloring} of a multigraph $G$ is a coloring $\phi: E(G) \to [k]$ such that if any two edges $e_1$ and $e_2$ are either adjacent to each other or adjacent to a common edge, then $\phi(e_1) \neq \phi(e_2)$.  In other words, the edges in each color class form an induced matching in the original multigraph.  The \emph{strong chromatic index} of $G$, denoted by $\chi'_s(G)$, is the minimum $k$ for which $G$ has a strong $k$-edge-coloring.  This is equivalent to finding the chromatic number of the square of the line graph of $G$.

Fouquet and Jolivet \cite{FJ1, FJ2} introduced the notion of strong edge-coloring, which was used to solve a problem involving radio networks and their frequencies.  
%For example, suppose we have transmitters $w,x,y,z$, where transmitters $x$ and $y$ are `close.' That is, their locations are such that messages sent or received by $x$ or $y$ will overlap if transmitted on the same frequency.  So, if $w$ and $x$ transmit messages across a frequency $\alpha$, $x$ and $y$ cannot use the same frequency.  Additionally, if $y$ and $z$ transmit messages across a frequency $\beta$, then as $x$ and $y$ are close, we do not want $\alpha$ and $\beta$ to be the same frequency.  Thus, if we think of the transmitters as the vertices of a graph and the edges correspond to vertices that are close and can transmit or receive messages, the assignment of frequencies is equivalent to finding a strong edge-coloring of the graph.  
More details on this application can be found in \cite{NKGB, R}.  %RL}. %BIKMTT }.

For general graphs, the greedy algorithm provides an upper bound on $\chi'_s$ of $2(\Delta - 1) + 2(\Delta - 1)^2 + 1$, where $\Delta$ denotes the maximum degree of the multigraph.  At a 1985 seminar in Prague, Erd\H{o}s and Ne\v{s}et\v{r}il  conjectured that in fact a stronger upper bound holds, which if true, is best possible (see \cite{Erdos, ErdosNesetril}).

\begin{conj}[Erd\H{o}s and Ne\v{s}et\v{r}il '85]
If $G$ is a graph with maximum degree $\Delta$, then 

$\chi'_s(G) \le 
\left\{ \begin{array}{ll}
\frac{5}{4}\Delta^2, & \textrm{ if } \Delta \textrm{ is even,}\\
\frac{5}{4}\Delta^2 - \frac{1}{2}\Delta + \frac{1}{4},& \textrm{ if } \Delta \textrm{ is odd}
\end{array}\right.$
\end{conj}

When $G$ has maximum degree at most 3, the conjecture was verified by Andersen \cite{Andersen}, who proved the conjecture for multigraphs,  and independently by Hor\'{a}k, Qing and Trotter \cite{HQT}.  %When $G$ has maximum degree at most 4, the best known upper bound is 22, due to Cranston \cite{Cranston}.  In particular, Andersen extended the result to multigraphs.
In general,  the problem remains open with the best known upper bound due to Molloy and Reed \cite{MolloyReed} using probabilistic techniques.\footnote{Recently, Bruhn and Joos \cite{BJ} claim to have improved this bound to $1.93\Delta^2$.}

\begin{thm*}[Molloy and Reed '97]
For large enough $\Delta$, every graph $G$ with maximum degree $\Delta$ has $\chi'_s(G) \le 1.998\Delta^2$.
\end{thm*}

Faudree et al. \cite{FGST} show that when restricted to planar multigraphs,  $\chi'_s(G) \le 4\Delta + 4\mu$, where $\mu$ denotes the maximum number of parallel edges connecting a pair of vertices in $G$.  Additionally, they show that for every positive integer $k \ge 2$, there exists a planar graph $G$ with $\Delta = k$ and $\chi'_s(G) = 4\Delta - 4$.  %When further restricted to outerplanar graphs, Hocquard, Ochem and Valicov \cite{HOV} show that if $G$ is an outerplanar graph with $\Delta(G) \ge 3$, then $\chi_s'(G) \le 3\Delta(G) - 3$.

Borodin and Ivanova \cite{BorodinIvanova} show that if a planar graph $G$ has maximum degree at most $\Delta$ and girth (i.e. the length of a shortest cycle) at least $40\lfloor \frac{\Delta}{2}\rfloor + 1$, then $\chi'_s(G) \le 2\Delta - 1$.  %Chang, Montassier, P\^{e}cher and Raspaud reduce the necessary girth to at least $10\Delta + 46$, which improves upon the previous result for $\Delta \ge 6$.

In regards to \emph{subcubic} graphs, i.e., graphs with maximum degree at most 3, Faudree et al. \cite{FGST} pose the following set of conjectures.

\begin{conj}[Faudree et al. '90]
Let $G$ be a subcubic graph.
\begin{enumerate}[label=\bfseries 2.\arabic*]
\item  $\chi'_s(G) \le 10$
\item  If $G$ is bipartite, then $\chi'_s(G) \le 9$
\item  If $G$ is planar, then $\chi'_s(G) \le 9$
\item  If $G$ is bipartite and the degree sum along every edge is at most 5, then $\chi'_s(G) \le 6$.
\item  If $G$ is bipartite with girth at least 6, then $\chi'_s(G) \le 7$.
\item  If $G$ is bipartite with large girth, then $\chi'_s(G) \le 5$.
\end{enumerate}
\end{conj}

Andersen \cite{Andersen}, and independently Hor\'{a}k, Qing and Trotter \cite{HQT}, proved Conjecture 2.1.  Conjecture 2.2 was verified by Steger and Yu \cite{StegerYu}.  Conjecture 2.4 was confirmed by Wu and Lin \cite{WuLin} and was generalized by Nakprasit and Nakprasit \cite{NN}. %N  
The previously mentioned result of Borodin and Ivanova \cite{BorodinIvanova} verified Conjecture 2.6 for planar graphs.  The authors know of no results which pertain to Conjecture 2.5.  

The purpose of this paper is to verify Conjecture 2.3.  That is, we prove the following theorem, which is best possible by considering the complement of the cycle of length six.

\begin{thm}\label{thm:conj}
Every subcubic, planar multigraph $G$ with no loops has $\chi'_s(G) \le 9$.
\end{thm}

The proof of this result yields a polynomial time algorithm in terms of the number of vertices that  will color any subcubic, planar multigraph using at most nine colors.  Theorem \ref{thm:conj} implies the following corollary.

\begin{cor}
Every subcubic, planar multigraph $G$ with no loops contains an induced matching of size at least $|E(G)|/9$.
\end{cor}

This corollary extends a result of Kang, Mnich and M\"{u}ller \cite{KMM} to loopless multigraphs.  Joos, Rautenbach and Sasse \cite{JRS} later showed that the above lower bound holds for all subcubic graphs, thus proving a conjecture of Henning and Rautenbach \cite{HR}.

Hocquard et al. \cite{HMRV} provide upper bounds on the strong chromatic index of subcubic graphs based on the maximum average degree.  These results, which strengthen those of Hocquard and Valicov \cite{HV}, provide stronger upper bounds on the strong chromatic index of subcubic planar graphs based on girth.  In addition, they prove Conjecture 2.3 for subcubic planar graphs with no induced $C_4$ or $C_5$.  This result verifies Conjecture 2.3 for subcubic planar graphs with girth at least six, a statement independently obtained by Hud\'{a}k et al. \cite{HLSS}.

We present our result as follows. In Section \ref{sec:prelim}, we provide the notation we will use along with preliminary results. The remaining sections assume the existence of a minimal counterexample.  Section \ref{struct1} contains basic properties of a minimal counterexample, including the fact that it has no cycles of length three or four.  The lemmas in Section \ref{struct2} will show that if a face has a 2-vertex on its boundary, then the face has length at least eight, and additionally, if two 2-vertices exist on a face, then the distance between them is at least five on the face.  Section \ref{struct3} contains two lemmas showing that every face of length five is surrounded by faces of length at least seven.  Lastly, Section \ref{sec:proof} contains a discharging proof based on the lemmas presented in Sections \ref{struct1}, \ref{struct2} and \ref{struct3}.

%%%%%%%%%%%%%%%%%%%%%%%%%%%%%%%%%%%%%%%%%%%

\section{Preliminaries and notation}\label{sec:prelim}

In the proof of Theorem \ref{thm:conj}, we will often remove vertices or edges from a minimal counterexample and obtain a strong edge-coloring of the remaining multigraph.  To aid us, we introduce some notation and preliminary facts that we will use in explanations.

We will use some lower case Greek letters, such as $\alpha,\beta, \gamma, \delta$, to denote arbitrary colors, and we will use $\phi, \sigma, \psi$ to denote colorings.  Also an $i$\emph{-vertex} is a vertex of degree $i$ in our multigraph, and a $j$\emph{-face} is a face of length $j$ in our plane multigraph.  An $i^+$\emph{-vertex} and $j^+$\emph{-face} is a vertex of degree at least $i$ and a face of length at least $j$, respectively.

A coloring of a multigraph $G$ is \emph{good}, if it is a strong edge-coloring of $G$ using at most 9 colors.  
A \emph{partial} coloring of a graph $G$ is a coloring of any subset of $E(G)$, and we say it is a \emph{good partial coloring} of $G$, if for any colored edges $e_1$ and $e_2$ that are either adjacent to each other or adjacent to a common edge, we have $e_1$ and $e_2$ receiving different colors.  Given edges $e, e'$ in $G$, we say that $e$ \emph{sees} $e'$ if either $e$ and $e'$ are adjacent, or there is another edge $e''$ adjacent to both $e$ and $e'$.  Additionally, we will also say that $e$ \emph{sees} a color $\alpha$, if $e$ sees an edge $e'$ for which $\phi(e') = \alpha$, where $\phi$ is a partial coloring.

Let $\phi$ be a good partial coloring of a graph $G$.  For $v \in V(G)$, let $\sU_\phi(v)$  denote the set of colors used on the edges incident to $v$.  For an uncolored edge $e \in E(G)$, let $A_\phi(e)$ to denote the set of colors that can be used on $e$ to extend $\phi$ to a new good partial coloring of $G$.  For adjacent vertices $u,v$, let $\Upsilon_\phi(u,v) := \sU_\phi(u) \setminus \{\phi(uv)\}$.  That is, $\Upsilon_\phi(u,v)$ denotes the set of colors used on edges incident to $u$ other than $uv$.   As $\phi$ is a good partial coloring, $\Upsilon_\phi(u,v)$ and $\Upsilon_\phi(v,u)$ are disjoint.  Often we will refer to only one partial coloring which will not be named.  In these cases we will suppress the subscripts in the above notations.

As mentioned, we will remove vertices and edges from a multigraph $G$ to obtain a good partial coloring, say $\phi$.  Often, we will consider $|A_\phi(e)|$ for every uncolored $e$ in $G$, in order to apply the well known result of Hall \cite{Hall} in terms of systems of distinct representatives.  

\begin{thm*}[Hall '35]
Let $A_1, \dots, A_n$ be $n$ subsets of a set $U$.  A system of distinct representatives of $\{A_1,\dots, A_n\}$ exists if and only if  for all $k, 1 \le k \le n$ and every choice of subcollection of size $k$, $\{A_{i_1}, \dots, A_{i_k}\}$, we have $|A_{i_1} \cup \dots \cup A_{i_k}| \ge k$.
\end{thm*}

This will give a coloring of the remaining uncolored edges such that for every pair of uncolored edges $e_1$ and $e_2$, they will receive distinct colors from $A_\phi(e_1)$ and $A_\phi(e_2)$, respectively.  Such an extension of $\phi$ is a good coloring of $G$ and yields the desired result.  Thus, when left in a situation in which we can apply Hall's Theorem, we will say that we obtain a good coloring of $G$ by \emph{SDR}.

%%%%%%%%%%%%%%%%%%%%%%%%%%%%%%%%%%%%%%%%

\section{Basic Properties}\label{struct1}

Everywhere below we assume $G$ to be a subcubic, planar multigraph contradicting Theorem \ref{thm:conj}.  Among all such counterexamples, we assume that $G$ has the fewest vertices, and over all such counterexamples, has the fewest edges.  $G$ is connected, as otherwise we can color each component by the minimality of $G$, and so obtain a good coloring of $G$.  As $G$ is planar, we assume $G$ to be a \emph{plane} multigraph in all the following statements.  That is, we consider $G$ together with an embedding of $G$ into the plane.

In this section, we will show several properties of $G$, including that $G$ is simple, has no small cycles and the distance between any two 2-vertices is at least three, a fact that we will strengthen in a later section.  Similar statements are proven in \cite{HMRV, HV, HLSS} while considering minimal counterexamples with different properties.

\begin{lemma}\label{muledges}
$G$ has no multiple edges, i.e., $G$ is a simple graph.
\end{lemma}
\begin{proof}
Suppose that $e$ is a parallel edge in $G$.  By the minimality of $G$, $G - e$ has a good coloring.  Since $e$ sees at most seven edges in $G$, we can extend this good coloring to $G$.
\end{proof}

\begin{lemma}\label{delta}
$G$ has minimum degree at least 2.
\end{lemma}
\begin{proof}
Suppose that $v$ is a $1$-vertex and $u$ is the neighbor of $v$. Then $G - v$ has a good coloring.  Since $uv$ sees at most six edges in $G$, we can extend this good coloring to $G$.
\end{proof}

\begin{lemma}\label{cut-vertex}
$G$ has no cut-vertex and no cut-edge.
\end{lemma}
\begin{proof}
Since $G$ is subcubic, the existence of a cut-vertex implies the existence of a cut-edge.  Thus, it suffices to suppose that $G$ has a cut-edge, say $v_1v_2$.  For $i = 1,2$, let $H_i$ be the component of $v_1v_2$ containing $v_i$.  By Lemma \ref{delta}, $|V(H_i)| \ge 2$.  Define $G_1$ to be the graph consisting of $H_1$ together with $v_2$ and the edge $v_1v_2$.  Similarly define 
$G_2$ to be the graph consisting of $H_2$ together with $v_1$ and the edge $v_1v_2$. 

By the minimality of $G$, $G_1$ and $G_2$ have good colorings, $\phi_1$ and $\phi_2$, respectively.  We may assume $\sU_{\phi_1}(v_1) \subseteq \{1,2,3\}$,  $\sU_{\phi_2}(v_2) \subseteq \{1,4,5\}$ with $\phi_1(v_1v_2) = \phi_2(v_1v_2) = 1$.  Merging these two colorings yields a good coloring of $G$.
\end{proof}

\begin{lemma}\label{2-edge-cut}
If $\{e_1, e_2\}$ is an edge-cut in $G$, then $e_1, e_2$ are adjacent to each other.
\end{lemma}
\begin{proof}
If not, then we have an edge-cut $\{u_{1}w_{1},u_{2}w_{2}\}$ in $G$ that is a matching.  We may assume that $u_1$ and $u_2$ are in the same component of $G - \{u_1w_1, u_2w_2\}$ so that we can define $H_u$ to be the component of $G - \{u_1w_1, u_2w_2\}$ containing $u_1$ and $u_2$.  Let $H_w = G - H_u$.  We may then let $G_u$ be the graph consisting of $H_u$ together with a new vertex $w$ whose neighborhood is $\{u_1,u_2\}$.  Similarly, let $G_w$ be the graph consisting of $H_w$ together with a new vertex $u$ whose neighborhood is $\{w_1,w_2\}$.  Observe that $G_u$ and $G_w$ are subcubic, planar multigraphs, and so by the minimality of $G$, $G_u$ and $G_w$ have good colorings $\phi_u$ and $\phi_w$, respectively.  

Now, if $|\sU_{\phi_w}(w_1) \cup \sU_{\phi_w}(w_2)| \le 5$, then we may assume that  $\sU_{\phi_w}(w_1) \cup \sU_{\phi_w}(w_2) \subseteq [5]$ with $uw_i$ being colored $i$.  Since $|\sU_{\phi_u}(u_1) \cup \sU_{\phi_u}(u_2)| \le 6$, we may similarly assume that $\sU_{\phi_u}(u_1) \cup \sU_{\phi_u}(u_2) \subseteq \{1,2,6,7,8,9\}$ with $wu_i$ being colored $i$.  We may then merge these two colorings to obtain a good coloring of $G$ in which $u_iw_i$ receives color $i$ for $i \in \{1,2\}$.

So, we  have $|\sU_{\phi_w}(w_1) \cup \sU_{\phi_w}(w_2)| = |\sU_{\phi_u}(u_1) \cup \sU_{\phi_u}(u_2)| = 6$.  This implies $u_1u_2, w_1w_2 \notin E(G)$.  Thus, we may assume that $\sU_{\phi_u}(u_1) = \{1,3,4\},  \sU_{\phi_w}(w_2) = \{2,3,4\}, \sU_{\phi_u}(u_2) = \{2,5,6\}, \sU_{\phi_w}(w_1) = \{1,5,6\}$ with $uw_i, wu_i$ being colored $i$.  Again, we can merge these two colorings to obtain a good coloring of $G$ in which $u_iw_i$ receives color $i$.
\end{proof}

\begin{lemma}\label{NoTriangle}
$G$ has no triangles.
\end{lemma}
\begin{proof}%
Suppose that $w_{0}w_{1}w_{2}$ is a triangle in $G$.  If $w_0$ is a 2-vertex, then as $G - w_{0}$ has a good coloring, and since each of $w_0w_1$ and $w_0w_2$ see at most colored 5 edges in $G$, we can extend this good coloring to $G$.  Thus, each $w_i$ is a 3-vertex, and we may assume $N_{G}(w_{0}) = \{u_{0}, w_{1}, w_{2}\}$, $N_{G}(w_{1}) = \{w_{0}, u_{1}, w_{2}\}$ and $N_{G}(w_{2}) = \{w_{0}, w_{1}, u_{2}\}$.  

Now, $G - \{w_0,w_1,w_2\}$ has a good coloring, which applied to $G$ is a good partial coloring such that $|A(w_iu_i)| \ge 3$ and $|A(w_iw_{i+1})| \ge 5$ for $i \in \{0,1,2\}$ taken modulo 3.  If there are at least six colors available on these six uncolored edges, then we can extend to a good coloring of $G$ by SDR.  So we may assume $A(w_0w_1) = A(w_1w_2) = A(w_2w_0)$ and $|A(w_0w_1)| = 5$.  Without loss of generality, we may assume $A(w_0w_1) = \{1,2,3,4,5\}$.  However, this implies that for $i \in \{0,1,2\}$,  $\sU(u_i)$ and $\sU(u_{i+1})$ partition $\{6,7,8,9\}$, which cannot happen.  
\end{proof}

\begin{lemma}\label{separating}
$G$ has no separating cycle of length four or five.
\end{lemma}

\begin{proof}
We first show that $G$ has no 4-cycle with a 2-vertex.  Suppose that $w_{1}w_{2}w_{3}w_{4}$ is a 4-cycle.   If $w_1$ is a 2-vertex, then $G - w_1$ has a good coloring, such that $|A(w_1w_2)|$, $|A(w_4w_1)|$ $\ge 2$, and we can extend this to a good coloring of $G$.  Thus, if $G$ has a 4-cycle, then each vertex of the cycle is a 3-vertex.  We will use this below to show that $G$ has no separating 4-cycle or 5-cycle.

If on the contrary, $G$ has a separating 4-cycle or 5-cycle, call it $C$.  By Lemma \ref{NoTriangle}, $C$ has no chords, and as $G$ is subcubic, each vertex of $C$ is incident to at most one edge not on $C$.   Since $\lfloor \frac{5}{2}\rfloor=  2$, by symmetry we may assume that there are at most two edges inside $C$ that are incident to vertices on $C$ (recall that $G$ is assumed to be embedded in the plane).  If there is exactly one such edge, then $G$ has a cut-edge, contradicting Lemma \ref{cut-vertex}.  So, we have two such edges, which are in fact cut-edges, and by Lemma \ref{2-edge-cut}, these edges  share a common endpoint, say $u$, inside of $C$.  Now, $u$ is a 2-vertex, as otherwise it would be a cut-vertex with a cut-edge.  However, $u$ together with the vertices of $C$ has either a triangle or a 4-cycle containing a 2-vertex, contradicting Lemma \ref{NoTriangle} or the above, respectively.  Thus, $G$ has no separating 4-cycle or 5-cycle.  
\end{proof}

\begin{lemma}\label{No4cycle}
$G$ has no 4-cycle.
\end{lemma}

\begin{proof}
Suppose that $x_0x_1x_2x_3$ is a 4-cycle in $G$.  By Lemma \ref{separating}, this cycle is a 4-face and as is shown in the proof of Lemma \ref{separating}, each $x_i$ is a 3-vertex.  As a result, we let $y_i$ denote the third neighbor of $x_i$, which is not on this 4-cycle.  By Lemmas \ref{NoTriangle} and \ref{separating}, the $y_i$'s are distinct and $y_0y_2, y_1y_3 \notin E(G)$.  Let $G'$ denote the graph obtained from $G$ by removing $x_0,x_1,x_2,x_3$ and adding the edge $y_0y_2$.  Observe that $G'$ is a subcubic, planar multigraph, and so by the minimality of $G$, $G'$ has a good coloring.  Ignoring $y_0y_2$, we have a good partial coloring of $G$ that we  extend by coloring $x_0y_0, x_2y_2$ with the same color that $y_0y_2$ received.  This extended coloring is a good partial coloring, and we will refer to it as $\phi$.  As $|A_\phi(x_1y_1)|, |A_\phi(x_3y_3)| \ge 2$, we can greedily color these two edges and obtain another good partial coloring, which we will call $\sigma$.  %We may assume that $\sigma(x_1y_1) \neq \sigma(x_3y_3)$.

Note that the edges of the 4-cycle are the only uncolored edges of $G$ under $\sigma$, and $|A_\sigma(x_ix_{i+1})| \ge 2$ for each $i \in \{0,1,2,3\}$ taken modulo 4.  Additionally $\sU_\sigma(y_0) \cap \sU_\sigma(y_2) = \{\sigma(x_0y_0)\}$.  So, without loss of generality, we may assume that $\sU_\sigma(y_0) \subseteq \{1,2,3\}$ and $\sU_\sigma(y_2) \subseteq \{1,4,5\}$.  

Suppose that $|A_\sigma(x_0x_1) \cup A_\sigma(x_2x_3)| = 2$ so that without loss of generality, $A_\sigma(x_0x_1) = A_\sigma(x_2x_3) = \{8,9\}$. This implies $$\sU_\sigma(y_0) \cup \sU_\sigma(y_1) \cup \{\sigma(x_3y_3)\} =  \sU_\sigma(y_2) \cup \sU_\sigma(y_3) \cup \{\sigma(x_1y_1)\}= [7]$$ and additionally $\Upsilon_\sigma(y_1,x_1) = \{4,5\}, \Upsilon_\sigma(y_3,x_3) = \{2,3\}$.  However, this implies $|A_\sigma(x_1x_2)|$, $|A_\sigma(x_3x_0)| \ge 4$, and we can obtain a good coloring of $G$ by coloring the edges $x_0x_1, x_2x_3$, $x_1x_2$, $x_3x_0$ in this order.

So we  have $|A_\sigma(x_0x_1) \cup A_\sigma(x_2x_3)| \ge 3$ and by symmetry $|A_\sigma(x_1x_2) \cup A_\sigma(x_3x_0)| \ge 3$.  We may assume that $|A_\sigma(x_0x_1) \cup A_\sigma(x_1x_2) \cup A_\sigma(x_2x_3) \cup A_\sigma(x_3x_0)| \le 3$, otherwise we can obtain a good coloring of $G$ by SDR.

Now, if $|A_\sigma(x_0x_1)| = 2$, then $\Upsilon_\sigma(y_0,x_0) = \{2,3\}$, and additionally, $2,3 \notin \sU_\sigma(y_1) \cup \{\sigma(x_3y_3)\}$.  Since $\sU_\sigma(y_2) \subseteq \{1,4,5\}$, we have $2,3 \in A_\sigma(x_1x_2)$, but $2,3 \notin A_\sigma(x_0x_1)$.  Thus, $|A_\sigma(x_0x_1) \cup A_\sigma(x_1x_2)| \ge 4$, a contradiction.  So, $|A_\sigma(x_0x_1)| = 3$, and by symmetric arguments, we have $A_\sigma(x_0x_1) = A_\sigma(x_1x_2) = A_\sigma(x_2x_3) = A_\sigma(x_3x_0)$.

If $\Upsilon_\sigma(y_0,x_0) \subseteq \sU_\sigma(y_1) \cup \{\sigma(x_3y_3)\}$, then $|A_\sigma(x_0x_1)| \ge 4$, a contradiction.  Thus, say $2 \notin \sU_\sigma(y_1) \cup \{\sigma(x_3y_3)\}$.  However, $2 \notin \sU_\sigma(y_2)$ so that $2 \in A_\sigma(x_1x_2)\setminus A_\sigma(x_0x_1)$, again a contradiction.  Thus, in all cases we can extend $\sigma$ and obtain a good coloring of $G$.
\end{proof}

\begin{lemma}\label{distance>=3}
The distance between any two 2-vertices is at least three.
\end{lemma}
\begin{proof}
Let $u, v$ be 2-vertices in $G$.  Suppose first that $u,v$ are adjacent, and let $w$ be the other neighbor of $v$, which is possibly the other neighbor of $u$ as well.  Now, $G - v$ has a good coloring, and since $uv$ sees at most 5 colored edges in $G$ and $vw$ sees at most seven colored edges in $G$, we can extend this good coloring to $G$.  Thus, $u$ and $v$ are at least distance two apart in $G$.

Now suppose $u$ and $v$ are distance two apart and are both incident to a 3-vertex $x$.  Let $N_{G}(u) = \{u', x\}$, $N_{G}(v) = \{v',x\}$ and $N_{G}(x) = \{u, v, x'\}$, where $u', v', x'$ are not necessarily distinct.  By the minimality of $G$, $G - \{u,v,x\}$ has a good coloring such that $uu', vv', xx'$ each see at most six different colors, and $ux, vx$ each see at most four different colors.  Thus, we can extend this good partial coloring to $G$ by coloring the edges $uu', vv', xx', ux, vx$ in this order.  
\end{proof}

\section{Faces Without 2-Vertices}\label{struct2}

In this section, we show that if a face has a 2-vertex, then that face must have length at least eight.  Additionally, if a face does have two 2-vertices on its boundary, then the distance between them along the face is at least five.

\begin{lemma}\label{No2on5cycle}
Every vertex of a 5-cycle in $G$ is a 3-vertex.
\end{lemma}
\begin{proof}
By Lemma \ref{separating}, it suffices to consider 5-faces.  Suppose on the contrary that $x_{1}x_{2}x_{3}x_{4}x_{5}$ is a 5-face in $G$ and $x_{5}$ is a $2$-vertex.   Lemma \ref{distance>=3} implies that each $x_i$ other than $x_5$ has a third neighbor $y_i$.   By Lemmas \ref{NoTriangle}, \ref{separating} and \ref{No4cycle}, these $y_i$ are distinct, not on our cycle and pairwise nonadjacent except for possibly $y_1y_4$.  

Let $G'$ denote the graph obtained from $G$ by removing $x_1,x_2,x_3,x_4,x_5$ and adding the edge $y_2y_4$.  Observe that $G'$ is a subcubic, planar multigraph, and so by the minimality of $G$, $G'$ has a good coloring.  Ignoring $y_2y_4$, we have a good partial coloring of $G$ that we can extend by coloring $x_4x_5, x_2y_2$ with the color of $y_2y_4$.  Call this good partial coloring, $\phi$.  Note that $|A_\phi(x_3y_3)|, |A_\phi(x_4y_4)|  \ge2$ so that we can color these two edges greedily to obtain a new good partial coloring $\sigma$.  

Now, $|A_\sigma(x_1y_1)|, |A_\sigma(x_2x_3)|, |A_\sigma(x_3x_4)| \ge 2, |A_\sigma(x_1x_2)| \ge 3$ and $|A_\sigma(x_5x_1)| \ge 5$.  If $A_\sigma(x_1y_1) \cap A_\sigma(x_3x_4) = \emptyset$, then we can extend this to a good coloring of $G$ by SDR.  So we can color $x_1y_1, x_3x_4$ with the same color, $\alpha$.  We can then color the remaining three uncolored edges by SDR.
\end{proof}

\begin{lemma}\label{distance>=4}
The distance between any two 2-vertices is at least four.
\end{lemma}
\begin{proof}
By Lemma \ref{distance>=3}, we may consider a path $x_1x_2x_3x_4x_5x_6$ such that $x_2, x_5$ are 2-vertices.  By Lemma \ref{distance>=3}, all other $x_i$ are 3-vertices, and so, we let $y_3, y_4$ be the third neighbors of $x_3,x_4$, respectively.  By Lemmas \ref{NoTriangle}, \ref{No4cycle}, \ref{separating} and \ref{No2on5cycle}, $y_3, y_4$ are distinct, not on this path and the only possible adjacency between these eight vertices other than those on the path and $x_3y_3, x_4y_4$, is $x_1x_6$.  However, regardless of the existence of $x_1x_6$, the following argument holds.

By the minimality of $G$, $G - \{x_2,x_3,x_4,x_5\}$ has a good coloring such that $|A (x_1x_2)|$, $|A(x_3y_3)|$, $|A(x_4y_4)|$, $|A(x_5x_6)| \ge 3$ and $|A(x_2x_3)|, |A(x_3x_4)|, |A(x_4x_5)| \ge 5$ (when $x_1x_6 \in E(G)$, then we get $|A(x_1x_2)|, |A(x_5x_6)| \ge 4$).

If there exists $\alpha\in A(x_{2}x_{3})\setminus A(x_{4}x_{5})$ (or if $|A(x_4x_5)| \ge 6$), then we can color $x_{2}x_{3}$ with $\alpha$  (or color $x_2x_3$ first) and then color $x_{1}x_{2}$, $x_{3}y_{3}$, $x_{4}y_{4}$, $x_{3}x_{4}$, $x_{5}x_{6}$, $x_{4}x_{5}$ in this order to obtain a good coloring of $G$.   So, we may assume that 
$|A(x_{4}x_{5})|=5$ and $A(x_{2}x_{3})=A(x_{4}x_{5})$.

If $ A(x_{1}x_{2})\cap A(x_{2}x_{3}) = \emptyset$, then we can color
$x_{5}x_{6},x_{4}x_{5},x_{4}y_{4},x_{3}y_{3},x_{3}x_{4}$, $x_{2}x_{3},   x_{1}x_{2}$ in this order to obtain a good coloring of $G$.  Thus, it  remains to consider the case when  $A(x_{2}x_{3})=A(x_{4}x_{5})$ and there exists some $\beta \in A(x_{1}x_{2}) \cap A(x_2x_3)$.   In this case, we color $x_1x_2$ and $x_4x_5$ with $\beta$ and then color $x_{5}x_{6},x_{4}y_{4},x_{3}y_{3},x_{3}x_{4}$,   $x_{2}x_{3}$ in this order to obtain a good coloring of $G$.  
\end{proof}

\begin{lemma}\label{distance>=5}
If the boundary of a face in $G$ contains a pair of 2-vertices, then the distance on the boundary between them is at least five.
\end{lemma}
\begin{proof}
By Lemma \ref{distance>=4}, any face contradicting the statement has length at least eight and contain a path $x_1x_2x_3x_4x_5x_6x_7$ such that $x_2$ and $x_6$ are 2-vertices.  By Lemma \ref{distance>=4}, all other $x_i$ are 3-vertices, and so, for $j \in \{3,4,5\}$ we let $y_j$ be the neighbor of $x_j$ other than $x_{j-1}, x_{j+1}$.   By Lemmas \ref{NoTriangle}, \ref{separating} and  \ref{No4cycle}, we have that $y_3,y_4, y_5$ are distinct, pairwise nonadjacent and not on this path.  By the same Lemmas, the only possible adjacencies between these ten vertices other than those on the path and $x_3y_3, x_4y_4, x_5y_5$, are $x_1y_5, x_7y_3$.  However, both edges cannot exist simultaneously and their existence will not affect the following argument.

Let $G'$ be obtained from $G$ by removing $x_2,x_3,x_4,x_5,x_6$ and adding the edge $y_3y_5$.  Observe that $G'$ is a subcubic, planar multigraph, and so by the minimality of $G$, $G'$ has a good coloring.  Ignoring $y_3y_5$, we have a good partial coloring of $G$ that we can extend by coloring $x_3y_3$ and $x_5y_5$ with the color of $y_3y_5$.  We will refer to this coloring as $\phi$.  Note that $|A_\phi(x_1x_2)|$, $|A_\phi(x_4y_4)|$, $|A_\phi(x_6x_7)| \ge 2$ and $|A_\phi(x_ix_{i+1})| \ge 4$ for $i \in \{2,3,4,5\}$.  From here we see that the existence of $x_1y_5$ does not affect coloring $x_1x_2$  as $\phi(x_5y_5)$  is already excluded from $A_\phi(x_1x_2)$ since $x_1x_2$  sees $x_3y_3$.  Symmetrically, the existence of $x_7y_3$ does not affect coloring $x_6x_7$ as $\phi(x_3y_3)$ is already excluded from $A_\phi(x_6x_7)$ since $x_6x_7$ sees $x_5y_5$.

If there exists $\alpha \in A_\phi(x_4x_5)\setminus A_\phi(x_2x_3)$ (or if $|A_\phi(x_2x_3)| \ge 5$), then we can color $x_4x_5$ with $\alpha$ (or color $x_4x_5$ first) and then color $x_{6}x_{7}$, $x_{4}y_{4}$, $x_{5}x_{6}$, $x_{3}x_{4}$, $x_{1}x_{2}$, $x_{2}x_{3}$ in this order to obtain a good coloring of $G$.  So, we may assume that  $|A_\phi(x_{2}x_{3})|=4$ and $A_\phi(x_{2}x_{3})=A_\phi(x_{4}x_{5})$.

If $A_\phi(x_1x_2) \cap A_\phi(x_4x_5) = \emptyset$ (and consequently, $A_\phi(x_1x_2) \cap A_\phi(x_2x_3) = \emptyset$), then we can color $x_{6}x_{7},x_{4}y_{4},x_{5}x_{6},x_{4}x_{5},x_{3}x_{4},   x_{2}x_{3},x_{1}x_{2}$ in this order to obtain a good coloring of $G$.  Thus, it  remains to consider the case when there exists some $\beta \in A(x_{1}x_{2})\cap A(x_{4}x_{5})$.  In this case we color $x_1x_2, x_4x_5$ with $\beta$ and then color $x_{6}x_{7}$, $x_{4}y_{4}$, $x_{5}x_{6}$, $x_{3}x_{4}$, $x_{2}x_{3}$ in this order to obtain a good coloring of $G$. 
\end{proof}

\begin{lemma}\label{No2on6cycle}
Every vertex of a 6-cycle in $G$ is a 3-vertex. 
\end{lemma}
\begin{proof}
Suppose that $G$ has a $6$-cycle $C$ given by $x_0x_1x_2x_3x_4x_5$ on which $x_0$ is a 2-vertex.  By Lemma \ref{distance>=4}, $x_0$ is the only 2-vertex of $C$. 

\begin{case}
$C$ is a separating 6-cycle.
\end{case}

By Lemmas \ref{NoTriangle}, \ref{separating} and \ref{No4cycle}, $C$ has no chords.  Just as in the proof of Lemma \ref{separating}, we may assume that $C$ has at most two edges inside $C$ that are incident to vertices on $C$.  If there is exactly one such edge, then $G$ has a cut-edge, contradicting Lemma \ref{cut-vertex}.  So, we  have two such edges, and by Lemma \ref{2-edge-cut} these edges  share a common endpoint, say $u$, inside of $C$.  Now, $u$ is a 2-vertex, else it is a cut-vertex with a cut-edge.  However, $u$ together with the vertices of $C$ contains either a triangle, a 4-cycle, or a 5-cycle containing a 2-vertex, contradicting Lemmas \ref{NoTriangle}, \ref{No4cycle}, \ref{separating}, or \ref{No2on5cycle}, respectively.  

\begin{case}
$C$ is not a separating 6-cycle.
\end{case}

Recall that $G$ is assumed to be embedded into the plane.  Thus $C$ must be the boundary of a 6-face.  As mentioned above, each $x_i$, other than $x_0$, is a 3-vertex and so has a third neighbor $y_i$.  We claim that these $y_i$'s are distinct, pairwise disjoint and not on $C$.  Indeed, if any $y_i$ was on $C$, we would create either a triangle or 4-cycle, contradicting Lemmas \ref{NoTriangle} and \ref{No4cycle}.  For $i \in [4]$, if $y_i = y_{i+1}$, we have a triangle contradicting Lemma \ref{NoTriangle}.  For $i \in \{1,2,3,5\}$ taken modulo 5, if $y_i = y_{i+2}$, we have a 4-cycle contradicting Lemma \ref{No4cycle}.  For $i \in \{1,2\}$, if $y_i = y_{i+3}$, then $y_ix_ix_{i+1}x_{i+2}x_{i+3}y_{i+3}$ is a separating 5-cycle contradicting Lemma \ref{separating}.  Thus, the $y_i$'s are distinct.  For $i \in [4]$, if $y_iy_{i+1} \in E(G)$, we have a 4-cycle contradicting Lemma \ref{No4cycle}.  For $i \in [3]$ if $y_iy_{i+2} \in E(G)$, we have a separating 5-cycle contradicting Lemma \ref{separating}.  If $y_5y_1 \in E(G)$, then $y_1x_1x_0x_5y_5y_1$ is a 5-cycle containing a 2-vertex contradicting Lemma \ref{No2on5cycle}.  For $i \in \{1,2\}$ if $y_iy_{i+3} \in E(G)$, then $y_ix_ix_{i+1}x_{i+2}x_{i+3}y_{i+3}y_i$ is a separating 6-cycle contradicting Case 1.  Thus, the $y_i$'s are pairwise disjoint.

Now, let $G'$ denote the plane graph obtained from $G$ by adding a new vertex $z$ inside the face bounded by $C$, deleting $x_0, \dots, x_5$, and adding the new edges $zy_1, zy_3, zy_4$.  Observe that $G'$ is a subcubic, planar graph, and so by the minimality of $G$, it has a good coloring $\phi$.  Ignoring $zy_1, zy_3, zy_4$, this yields a good partial coloring of $G$ that can be extended by coloring $x_1y_1$ and $x_3x_4$  with $\phi(zy_1)$.  This coloring, call it $\sigma$, is indeed a good partial coloring as $\phi(zy_1)$ cannot appear in $\Upsilon_\phi(y_3,x_3) \cup \Upsilon_\phi(y_4,x_4)$ since $\phi$ was a partial good coloring.

Without loss of generality, suppose $\sigma(x_1y_1) = \sigma(x_3x_4) = 1$.  Note that $|A_\sigma(x_iy_i)| \ge 2$ for $i \in \{2,3,4,5\}$, $|A_\sigma(x_jx_{j+1})| \ge 4$ for $j \in \{1,2,4\}$ and $|A_\sigma(x_\ell x_{\ell+1})| \ge 6$ for $\ell \in \{0,5\}$ taken modulo 6.  As a result, if we can extend $\sigma$ to a good partial coloring on the edges $x_2y_2, x_3y_3, x_4y_4, x_5y_5, x_2x_3, x_4x_5$, then we can extend this further by coloring $x_1x_2, x_0x_1, x_0x_5$ in this order to obtain a good coloring of $G$.  Thus, it suffices to consider the edges $x_2y_2, x_3y_3, x_4y_4, x_5y_5, x_2x_3, x_4x_5$.

For $i \in \{2,3,4,5\}$, if there exists $\alpha \in A_\sigma(x_iy_i) \setminus A_\sigma(x_2x_3)$ (or $|A_\sigma(x_2x_3)| \ge 5$), then we can color $x_iy_i$ with $\alpha$ (or color $x_iy_i$ first).  If $i = 2$, we color $x_3y_3, x_4y_4, x_5y_5, x_4x_4, x_2x_3$ in this order.  If $i = 5$, we color $x_4y_4, x_3y_3, x_2y_2, x_4x_5, x_2x_3$ in this order.  If $i \in \{2,3\}$, we color $x_{i-1}y_{i-1},\dots, x_2y_2, x_{i+1}y_{i+1},\dots, x_5y_5$, $x_4x_5, x_2x_3$ in this order.  In all cases, we obtain our good partial coloring of $G$.  As a consequence, $|A_\sigma(x_2x_3)|$ $= 4$ and $A_\sigma(x_iy_i) \subseteq A_\sigma(x_2x_3)$ for $i \in \{2,3,4,5\}$.  By a symmetric argument, $|A_\sigma(x_4x_5)|$ $= 4$ and $A_\sigma(x_iy_i) \subseteq A_\sigma(x_4x_5)$.

Now, if there exists $\beta \in A_\sigma(x_3y_3) \setminus A_\sigma(x_2y_2)$ (or $|A_\sigma(x_2y_2)| \ge 3$), then we can color $x_3y_3$ with $\beta$ (or color $x_3y_3$ first) and then color $x_{4}y_{4}$, $x_{5}y_{5}$, $x_{4}x_{5}$, $x_2x_3$, $x_2y_2$ in this order to obtain our good partial coloring of $G$.  So we may assume that $|A_\sigma(x_2y_2)|=2$ and $A_\sigma(x_2y_2) = A_\sigma(x_3y_3)$. A similar argument shows that $|A_\sigma(x_{5}y_{5})|=2$ and $A_\sigma(x_{5}y_{5})= A_\sigma(x_{4}y_{4})$. 

Lastly, if there exists $\gamma  \in A_\sigma(x_2y_2) \cap A_\sigma(x_4y_4)$, then we can color $x_2y_2, x_4y_4$ with $\gamma$ and then color $x_3y_3$, $x_{5}y_{5}$, $x_{4}x_{5}$, $x_2x_3$ in this order to obtain our good partial coloring of $G$.  

Thus, $A_\sigma(x_2y_2) = A_\sigma(x_3y_3)$ and $A_\sigma(x_4y_4) = A_\sigma(x_5y_5)$.  Furthermore, $A_\sigma(x_2y_2)$ and $A_\sigma(x_4y_4)$ partition $A_\sigma(x_2x_3)$ and $A_\sigma(x_4x_5)$ so that $A_\sigma(x_2x_3) = A_\sigma(x_4x_5)$.  So without loss of generality, we may assume that $A_\sigma(x_2y_2)=  A_\sigma(x_3y_3)=\{2,3\}$, $A_\sigma(x_{4}y_{4})=  A_\sigma(x_{5}y_{5})=\{4,5\}$ and $A_\sigma(x_2x_3)=  A_\sigma(x_{4}x_{5})=\{2,3,4,5\}$.  We can then obtain a good partial coloring of $G$ by coloring $x_iy_i$ with $i$ for $i \in \{2,3,4,5\}$, $x_2x_3$ with 5 and $x_4x_5$ with 2.  As mentioned above, these good partial colorings can each be extended to obtain good colorings of $G$.

This completes the case that $C$ is the boundary of a 6-face, and so proves the lemma.
\end{proof}

\begin{figure}[h]
\centering
\begin{tikzpicture}[line cap=round,line join=round,>=triangle 45,x=0.75cm,y=0.75cm]
\clip(-9, -4) rectangle (9, 3);
\draw (-5.0,2.0)-- (-6.5630409448376374,1.2477591934187724);
\draw (-6.5630409448376374,1.2477591934187724)-- (-6.950298708611126,-0.44309699523893453);
\draw (-6.5630409448376374,1.2477591934187724)-- (-8.126081889675275,2.495518386837545);
\draw (-6.950298708611126,-0.44309699523893453)-- (-5.870462198744655,-1.8006375428043868);
\draw (-6.950298708611126,-0.44309699523893453)-- (-8.900597417222253,-0.8861939904778691);
\draw (-5.0,2.0)-- (-3.4369590551623626,1.2477591934187724);
\draw (-3.4369590551623626,1.2477591934187724)-- (-3.0497012913888737,-0.44309699523893453);
\draw (-3.4369590551623626,1.2477591934187724)-- (-1.873918110324725,2.495518386837545);
\draw (-3.0497012913888737,-0.44309699523893453)-- (-4.129537801255345,-1.8006375428043868);
\draw (-3.0497012913888737,-0.44309699523893453)-- (-1.0994025827777474,-0.8861939904778691);
\draw (-5.870462198744655,-1.8006375428043868)-- (-4.129537801255345,-1.8006375428043868);
\draw (-5.870462198744655,-1.8006375428043868)-- (-6.74092439748931,-3.6012750856087736);
\draw (-4.129537801255345,-1.8006375428043868)-- (-3.2590756025106904,-3.6012750856087736);
\draw (-5.4,1.9) node[anchor=north west] {$x_0$};
\draw (-4.2,1.4) node[anchor=north west] {$x_1$};
\draw (-3.9,0) node[anchor=north west] {$x_2$};
\draw (-4.7,-1.1) node[anchor=north west] {$x_3$};
\draw (-6,-1.1) node[anchor=north west] {$x_4$};
\draw (-6.9,0) node[anchor=north west] {$x_5$};
\draw (-6.6,1.4) node[anchor=north west] {$x_6$};
\draw (-2.7,3.1) node[anchor=north west] {$y_1$};
\draw (-1.9,0) node[anchor=north west] {$y_2$};
\draw (-3.4,-2.8) node[anchor=north west] {$y_3$};
\draw (-7.4,-2.8) node[anchor=north west] {$y_4$};
\draw (-8.9, 0) node[anchor=north west] {$y_5$};
\draw (-8,3.1) node[anchor=north west] {$y_6$};
\draw (7.3,3.1) node[anchor=north west] {$y_1$};
\draw (8.1,0) node[anchor=north west] {$y_2$};
\draw (6.6,-2.8) node[anchor=north west] {$y_3$};
\draw (2.6,-2.8) node[anchor=north west] {$y_4$};
\draw (1.1,0) node[anchor=north west] {$y_5$};
\draw (2.,3.1) node[anchor=north west] {$y_6$};
\draw (1.87, 2.49)--(8.12,2.49);
\draw (8.9, -0.88)--(3.25, -3.6);
\draw [->] (-1.0,0.0) -- (1.0,0.0);
\begin{scriptsize}
\draw [fill=black] (-5.0,2.0) circle (1.5pt);
\draw [fill=black] (-6.5630409448376374,1.2477591934187724) circle (1.5pt);
\draw [fill=black] (-6.950298708611126,-0.44309699523893453) circle (1.5pt);
\draw [fill=black] (-5.870462198744655,-1.8006375428043868) circle (1.5pt);
\draw [fill=black] (-6.5630409448376374,1.2477591934187724) circle (1.5pt);
\draw [fill=black] (-6.950298708611126,-0.44309699523893453) circle (1.5pt);
\draw [fill=black] (-6.5630409448376374,1.2477591934187724) circle (1.5pt);
\draw [fill=black] (-8.126081889675275,2.495518386837545) circle (1.5pt);
\draw [fill=black] (-6.950298708611126,-0.44309699523893453) circle (1.5pt);
\draw [fill=black] (-5.870462198744655,-1.8006375428043868) circle (1.5pt);
\draw [fill=black] (-6.950298708611126,-0.44309699523893453) circle (1.5pt);
\draw [fill=black] (-8.900597417222253,-0.8861939904778691) circle (1.5pt);
\draw [fill=black] (-6.950298708611126,-0.44309699523893453) circle (1.5pt);
\draw [fill=black] (-5.870462198744655,-1.8006375428043868) circle (1.5pt);
\draw [fill=black] (-5.0,2.0) circle (1.5pt);
\draw [fill=black] (-3.4369590551623626,1.2477591934187724) circle (1.5pt);
\draw [fill=black] (-5.0,2.0) circle (1.5pt);
\draw [fill=black] (-3.4369590551623626,1.2477591934187724) circle (1.5pt);
\draw [fill=black] (-3.0497012913888737,-0.44309699523893453) circle (1.5pt);
\draw [fill=black] (-3.4369590551623626,1.2477591934187724) circle (1.5pt);
\draw [fill=black] (-1.873918110324725,2.495518386837545) circle (1.5pt);
\draw [fill=black] (-3.0497012913888737,-0.44309699523893453) circle (1.5pt);
\draw [fill=black] (-4.129537801255345,-1.8006375428043868) circle (1.5pt);
\draw [fill=black] (-3.0497012913888737,-0.44309699523893453) circle (1.5pt);
\draw [fill=black] (-1.0994025827777474,-0.8861939904778691) circle (1.5pt);
\draw [fill=black] (-3.0497012913888737,-0.44309699523893453) circle (1.5pt);
\draw [fill=black] (-4.129537801255345,-1.8006375428043868) circle (1.5pt);
\draw [fill=black] (-3.0497012913888737,-0.44309699523893453) circle (1.5pt);
\draw [fill=black] (-4.129537801255345,-1.8006375428043868) circle (1.5pt);
\draw [fill=black] (-5.870462198744655,-1.8006375428043868) circle (1.5pt);
\draw [fill=black] (-6.74092439748931,-3.6012750856087736) circle (1.5pt);
\draw [fill=black] (-5.870462198744655,-1.8006375428043868) circle (1.5pt);
\draw [fill=black] (-5.870462198744655,-1.8006375428043868) circle (1.5pt);
\draw [fill=black] (-5.870462198744655,-1.8006375428043868) circle (1.5pt);
\draw [fill=black] (-5.870462198744655,-1.8006375428043868) circle (1.5pt);
\draw [fill=black] (-4.129537801255345,-1.8006375428043868) circle (1.5pt);
\draw [fill=black] (-3.2590756025106904,-3.6012750856087736) circle (1.5pt);
\draw [fill=black] (-4.129537801255345,-1.8006375428043868) circle (1.5pt);
\draw [fill=black] (-4.129537801255345,-1.8006375428043868) circle (1.5pt);
\draw [fill=black] (-4.129537801255345,-1.8006375428043868) circle (1.5pt);
\draw [fill=black] (-4.129537801255345,-1.8006375428043868) circle (1.5pt);
\draw [fill=black] (1.873918110324725,2.495518386837545) circle (1.5pt);
\draw [fill=black] (8.126081889675275,2.495518386837545) circle (1.5pt);
\draw [fill=black] (8.900597417222253,-0.8861939904778691) circle (1.5pt);
\draw [fill=black] (1.0994025827777474,-0.8861939904778691) circle (1.5pt);
\draw [fill=black] (3.2590756025106904,-3.6012750856087736) circle (1.5pt);
\draw [fill=black] (6.74092439748931,-3.6012750856087736) circle (1.5pt);
\end{scriptsize}
\end{tikzpicture}
\caption{Forming $G'$ from $G$}
\label{fig:No2on7face}
\end{figure}
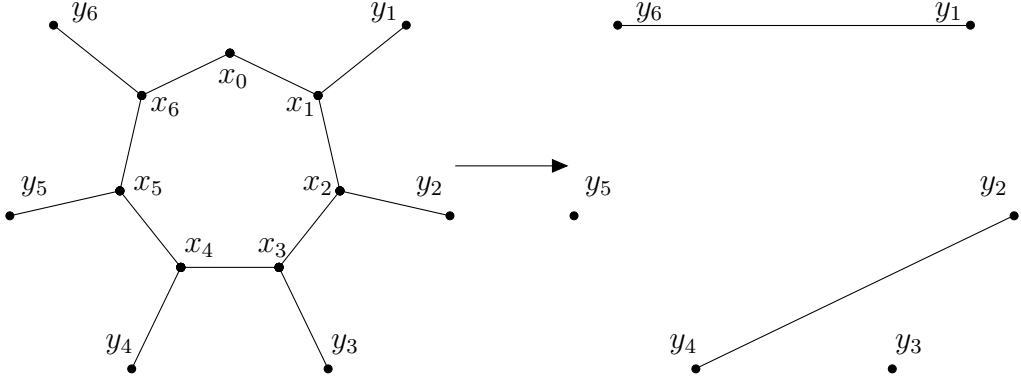

\begin{lemma}\label{No2on7face}
Every vertex of a 7-face in $G$ is a 3-vertex.
\end{lemma}

\begin{proof}
Recall that $G$ is assumed to be embedded into the plane.  Suppose on the contrary that $G$ has a 7-face with boundary $x_0x_1x_2\dots x_6$ with $x_0$ being a 2-vertex.  By Lemma \ref{distance>=5}, each $x_i$ other than $x_0$ has a third neighbor $y_i \notin \{x_{i-1},x_{i+1}\}$ where $i$ is taken modulo 7.  Similarly to  Case 2 of Lemma \ref{No2on6cycle},  Lemmas \ref{NoTriangle}, \ref{No4cycle}, \ref{separating}, \ref{No2on5cycle} and \ref{No2on6cycle}, imply that the $y_i$'s are not on the 7-face, are distinct and the only possible adjacencies other than those on this face or $x_iy_i$, $i \in [6]$, are $y_1y_4, y_2y_5, y_3y_6$.  Note by Lemma \ref{No2on6cycle},  $y_2y_6, y_1y_5 \notin E(G)$.

Let $G'$ be obtained from $G$ by removing $x_0, x_1, \dots, x_6$ and adding the edges $y_1y_6, y_2y_4$ (see Figure \ref{fig:No2on7face}).  Observe that $G'$ is a subcubic, planar multigraph, and so by the minimality of $G$, $G'$ has a good coloring, which ignoring $y_1y_6, y_2y_4$, is a good partial coloring $\phi$ of $G$.

\begin{claim}\label{cl1}
$A_\phi(x_2y_2)\cap A_\phi(x_{4}y_{4})\cap A_\phi(x_6y_6)=\emptyset$.
\end{claim}

\begin{proof}
Without loss of generality, suppose on the contrary that $1 \in A_\phi(x_2y_2) \cap A_\phi(x_4y_4) \cap A_\phi(x_6y_6)$.  We can obtain another good partial coloring of $G$, $\sigma$, by coloring $x_2y_2, x_4y_4, x_6y_6$ with 1.  Recall that $y_iy_{i+3}$, $i \in [3]$ are possible edges of $G$.  However, the existence of these edges will not affect the following argument as we will be sure to not color $x_1y_1, x_3y_3, x_5y_5$ with 1.

Note that $|A_\sigma(x_iy_i)| \ge 2$ for  $i \in \{1,3,5\}$, $|A_\sigma(x_jx_{j+1})| \ge 4$ for $j \in [5]$ and $|A_\sigma(x_6x_0)|$, $|A_\sigma(x_0x_1)| \ge 6$.  As a result, if we can somehow extend $\sigma$ to a good partial coloring on the edges $x_{1}y_{1}$, $x_{3}y_{3}$, $x_5y_{5}$, $x_1x_2$, $x_{2}x_{3}$, $x_{3}x_{4}$, $x_{4}x_{5}$, then we can extend this further by coloring $x_{5}x_{6}$, $x_{6}x_{0}$, $x_{0}x_{1}$ in this order.  Thus, it suffices to consider the edges $x_{1}y_{1}, x_{3}y_{3}, x_5y_{5},x_1x_2$, $x_{2}x_{3}$, $x_{3}x_{4}$, $x_{4}x_{5}$.

Now, if there exists $\alpha \in A_\sigma(x_2x_3)\setminus A_\sigma(x_4x_5)$ (or $|A_\sigma(x_4x_5)| \ge 5$), we can color $x_2x_3$ with $\alpha$ (or just color $x_2x_3$ first) and then color $x_1y_1, x_3y_3, x_1x_2, x_3x_4, x_5y_5, x_4x_5$ in this order to obtain our good partial coloring of $G$.  So, we may assume that $|A_\sigma(x_4x_5)| = 4$ and $A_\sigma(x_4x_5) = A_\sigma(x_2x_3)$.

If $A_\sigma(x_5y_5) \cap A_\sigma(x_2x_3) = \emptyset$ (and consequently, $A_\sigma(x_5y_5) \cap A_\sigma(x_4x_5) = \emptyset$), then we can color $x_1y_1,x_{3}y_{3},x_{1}x_2,x_{2}x_{3},x_3x_{4},   x_4x_5,x_5y_5$ in this order to obtain our good partial coloring of $G$.  Thus, it remains to consider the case when there exists some $\beta \in A_\sigma(x_5y_5) \cap A_\sigma(x_2x_3)$.  In this case, we color $x_5y_5, x_2x_3$ with $\beta$ and then color $x_1y_1, x_3y_3, x_1x_2, x_3x_4, x_4x_5$ in this order to obtain a good coloring of $G$.  This proves the claim.
\end{proof}

Recall that we originally constructed the auxiliary graph $G - \{x_0,\dots, x_6\} + y_1y_6 + y_2y_4$ to obtain $\phi$.  By Claim \ref{cl1}, the colors placed on $y_1y_6, y_2y_4$ are distinct, as they are colors in $A_\phi(x_6y_6)$ and $A_\phi(x_2y_2) \cap  A_\phi(x_4y_4)$, respectively.  So we may assume that $y_1y_6$ and $y_2y_4$ received the colors 1 and 2, respectively.  

Coloring $x_1y_1, x_6y_6$ with 1 and $x_2y_2, x_4y_4$ with 2, extends $\phi$ to a good partial coloring of $G$.  Additionally, under this new partial coloring, $x_5y_5$ sees at most eight colored edges, including edges colored 1 and 2, so that we can extend further by coloring $x_5y_5$ with some $\alpha$.  We will refer to this new good partial coloring in which $x_1y_1, x_6y_6$ are colored 1, $x_2y_2, x_4y_4$ are colored 2 and  $x_5y_5$ is colored $\alpha$, as $\psi$.  

Under $\psi$, the existence of $y_1y_4, y_2y_5$ will not affect our arguments as the edges $x_1y_1, x_4y_4$, $x_2y_2$, $x_5y_5$ are already colored in a good partial coloring.  The existence of the edge $y_3y_6$ will not affect our arguments as we will not color $x_3y_3$ with 1.  

Observe that $|A_\psi(x_3y_3)|, |A_\psi(x_4x_5)|, |A_\psi(x_5x_6)| \ge 2$, $|A_\psi(x_ix_{i+1})| \ge 3$ for $i \in [3]$ and $|A_\psi(x_6x_0)|, |A_\psi(x_0x_1)| \ge 5$.   As a result, if we can somehow extend $\psi$ to a good partial coloring on the edges $x_1x_2, x_2x_3, x_3x_4, x_4x_5, x_5x_6, x_3y_3$, then we can extend this further by coloring $x_0x_1, x_6x_0$.  Thus, it suffices to consider the edges $x_1x_2, x_2x_3, x_3x_4, x_4x_5, x_5x_6, x_3y_3$ below.

\begin{claim}\label{cl3}
$A_\psi(x_4x_5) = A_\psi(x_5x_6)$ and $|A_\psi(x_4x_5)| = 2$.
\end{claim}

\begin{proof}
Suppose on the contrary that either $|A_\psi(x_5x_6)| \ge 3$ or $A_\psi(x_4x_5) \setminus A_\psi(x_5x_6) \neq \emptyset$.  In either case, we color $x_4x_5$ first, where in the latter case we use a color from $A_\psi(x_4x_5)\setminus A_\psi(x_5x_6)$.  Suppose that $\beta$ is the color we can apply to $x_4x_5$.  Note that there exists some $\gamma_1 \in  A_\psi(x_3y_3)\setminus\{\beta\}$ as an available color for $x_3y_3$.  

We aim to show that it is impossible for $A_\psi(x_2x_3) = A_\psi(x_3x_4) = \{\beta, \gamma_1, \gamma_2\}$ for some $\gamma_2 \notin \{\beta, \gamma_1\}$.  If this was the case, then as $1,2 \notin A_\psi(x_2x_3)$, we may assume that $\beta = 3, \gamma_1 = 4$ and $\gamma_2 = 5$.  Additionally, as $\alpha \notin \{\beta, \gamma_1, \gamma_2\}$, we may assume that $\alpha = 6$.    Thus, we  have $\Upsilon_\psi(y_3,x_3) \cup \Upsilon_\psi(y_4,x_4) = \{1,7,8,9\}$ and $\Upsilon_\psi(y_2,x_2) \cup \Upsilon_\psi(y_3,x_3) = \{6,7,8,9\}$.  This implies that $\Upsilon_\psi(y_2,x_2) \cap \Upsilon_\psi(y_4,x_4) \neq \emptyset$. However, recall that the auxiliary graph used to obtain $\phi$ contained $y_2y_4$.  As a result, $\Upsilon_\psi(y_2,x_2) \cap \Upsilon_\psi(y_4,x_4) = \emptyset$, a contradiction.  So we cannot have $A_\psi(x_2x_3) = A_\psi(x_3x_4) = \{\beta, \gamma_1, \gamma_2\}$, as desired.

As a result, if we color $x_4x_5$ with $\beta$ and $x_3y_3$ with $\gamma_1$, we can further color $x_2x_3, x_3x_4$ to obtain a good partial coloring of $G$, which we will call $\tau$.  Let $\gamma_2, \gamma_3 $ denote $\tau(x_2x_3), \tau(x_3x_4)$, respectively.  Without loss of generality, we may assume $\gamma_1 = 7, \gamma_2 = 8, \gamma_3 = 9$.   Recall that we are assuming either $|A_\psi(x_5x_6)| \ge 3$ or $\beta \in A_\psi(x_4x_5)\setminus A_\psi(x_5x_6)$.  So $A_\tau(x_5x_6) \neq \emptyset$, and if $A_\tau(x_1x_2) \neq \emptyset$, we can greedily color $x_1x_2, x_5x_6$ to obtain a good partial coloring which we can extend to all of $G$ as mentioned above.  

Thus, we  had $A_\psi(x_1x_2) = \{7,8,9\}$.  We may also assume that $\sU_\psi(y_1) = \{1,3,4\}$ and $\sU_\psi(y_2) = \{2,5,6\}$.  Under $\tau$, if we could recolor $x_2x_3$ with either 3 or 4, then we could color $x_1x_2$ with 8 and color $x_5x_6$ last to obtain our good partial coloring of $G$.  Thus, $3,4 \in \Upsilon_\tau(y_3,x_3) \cup \{\beta\}$.  A similar argument holds if we could recolor $x_3x_4$ with 1, 3, 4, 5, or 6, implying $1,3,4,5,6 \in \Upsilon_\tau(y_3,x_3) \cup \Upsilon_\tau(y_4,x_4) \cup \{\alpha,\beta\}$.  

Recall that $y_2y_4$ was an edge of $G'$ so that $\Upsilon_\tau(y_2,x_2) \cap \Upsilon_\tau(y_4,x_4) = \emptyset$.  In particular, $5,6 \notin \Upsilon_\tau(y_4,x_4)$.  Thus, we have $5,6 \in \Upsilon_\tau(y_3,x_3) \cup \{\alpha,\beta\}$, and consequently, $\Upsilon_\tau(y_3,x_3) \cup \{\alpha,\beta\} = \{3,4,5,6\} = \Upsilon_\tau(y_1,x_1) \cup \Upsilon_\tau(y_2,x_2)$, and $1 \in \Upsilon_\tau(y_4,x_3)$.  

Let us reconsider $\psi$.  As $1 \in \Upsilon_\psi(y_4,x_3)$, we have $|A_\psi(x_4x_5)| \ge 3$.  If either $|A_\psi(x_5x_6)| \ge 3$ or $|A_\psi(x_4x_5)\setminus A_\psi(x_5x_6)| \ge 2$, then instead of coloring $x_4x_5$ with $\beta$, we could color it with some $\beta' \neq \beta$ such that $x_5x_6$ would still have at least two colors available on it.  By repeating an argument similar to the above, we would then conclude that $\Upsilon_\tau(y_3,x_3) \cup \{\alpha, \beta'\} = \Upsilon_\tau(y_1,x_1) \cup \Upsilon_\tau(y_2,x_2)$, a contradiction, as it would imply $\beta = \beta'$.  

As a result, we have $|A_\psi(x_5x_6)| = 2$ and $|A_\psi(x_4x_5)\setminus A_\psi(x_5x_6)| = 1$.  We may assume that $A_\psi(x_5x_6) = \{\delta_1, \delta_2\}$ and $A_\psi(x_4x_5) = \{\beta, \delta_1, \delta_2\}$.  Recall that $\Upsilon_\psi(y_3,x_3) \cup \{\alpha,\beta\} = \{3,4,5,6\}$ so that $\beta \notin \Upsilon_\psi(y_3,x_3)$, and consequently, $\beta \in A_\psi(x_3x_4)$.  

If $\{\delta_1, \delta_2\} \neq \{7,8\}$, then we can color $x_4x_5$ with a color in $\{\delta_1, \delta_2\}\setminus \{7,8\}$, color $x_3x_4$ with $\beta$, $x_3y_3$ with 7, $x_2x_3$ with 8, $x_1x_2$ with 9 and color $x_5x_6$ last to obtain our good partial coloring of $G$.  If $\{\delta_1, \delta_2\} = \{7,8\}$, then we can color $x_1x_2, x_4x_5$ with 8 and  $x_3y_3, x_5x_6$ with 7.  This good partial coloring of $G$ leaves at least one available color on each of $x_2x_3, x_3x_4$.  In particular, 5 and 6 are not available on $x_2x_3$.  If 5 or 6 is in $\Upsilon_\psi(y_3,x_3)$, then $x_2x_3$ has at least two available colors and we obtain our good partial coloring of $G$.  Since we cannot have 5 or 6 in $\Upsilon_\psi(y_4,x_4)$, we must have either 5 or 6 available on $x_3x_4$.  Thus, we can color $x_3x_4, x_2x_3$ and obtain our good partial coloring of $G$.

As mentioned above, these good partial colorings of $G$ can be extended to good colorings of $G$, and this proves the claim.
\end{proof}

Without loss of generality suppose $\alpha = 3$.  As $1,2,3 \notin A_\psi(x_4x_5)$, we may assume that $A_\psi(x_4x_5) = A_\psi(x_5x_6) = \{8,9\}$.  Additionally, we may assume that $\Upsilon_\psi(y_6,x_6) = \{4,5\} = \Upsilon_\psi(y_4,x_4)$ and $\Upsilon_\psi(y_5,x_5) = \{6,7\}$.  If $1 \in A_\psi(x_3x_4)$, we can color $x_3x_4$ with 1 and then color $x_3y_3, x_4x_5, x_5x_6, x_2x_3, x_1x_2$ in this order to obtain our good partial coloring of $G$.  Thus, $1 \in \Upsilon_\psi(y_3,x_3)$, and so $|A_\psi(x_2x_3)| \ge 4$.

Recall that $|A_\psi(x_3x_4)| \ge 3$, and thus, $x_3x_4$ has an available color not in $\{8,9\}$.  As $1,2,3,4,5 \notin A_\psi(x_3x_4)$, we may assume without loss of generality that it is 6.  So, we color $x_3x_4$ with 6 and then color $x_3y_3, x_4x_5, x_5x_6, x_2x_3$ in this order.  Call this good partial coloring of $G$, $\tau$.  It remains only to color $x_1x_2$ to obtain a good partial coloring of $G$ that we can extend to all of $G$.  

We must have $A_\psi(x_1x_2) = \{6, \tau(x_2x_3), \tau(x_3y_3)\}$, otherwise we can color $x_1x_2$.  Recall that our auxiliary graph $G'$ contained the edges $y_1y_6, y_2y_4$ so that $\Upsilon_\psi(y_1,x_1) \cap \Upsilon_\psi(y_6,x_6) = \Upsilon_\psi(y_2,x_2) \cap \Upsilon(y_4,x_4) = \emptyset$.  Since $\Upsilon_\psi(y_4,x_4) = \Upsilon(y_6,x_6) = \{4,5\}$, we have $4,5 \in A_\psi(x_1x_2)$, and in particular, $A_\psi(x_1x_2) = \{4,5,6\}$ with $\{\tau(x_2x_3), \tau(x_3y_3)\} = \{4,5\}$.

Without loss of generality assume $\tau(x_3y_3) = 4$.  We may then extend $\psi$ by coloring $x_3x_4$ with 6, $x_3y_3$ with 4, $x_1x_2$ with 5 and then color $x_2x_3, x_4x_5, x_5x_6$ in this order to obtain our good partial coloring of $G$.

In all cases, we obtain a partial good coloring of $G$ from which we can extend to a good coloring of $G$ as mentioned above.  This proves the lemma. 
\end{proof}

\section{Adjacent Faces}\label{struct3}

By the lemmas in Section \ref{struct1}, every face in $G$ is a $5^+$-face.  In this section we show that if a face has length five, then it can only be adjacent to $7^+$-faces.

\begin{figure}[h]
\centering
\begin{tikzpicture}[line cap=round,line join=round,>=triangle 45,x=0.75cm,y=0.75cm]
\clip(-10, -6) rectangle (10,6);
\draw [->] (-1.0,0.0) -- (1.0,0.0);
\draw (-6.0,5.62)-- (-6.0,3.62);
\draw (-6.0,3.62)-- (-7.902113032590307,2.238033988749895);
\draw (-9.804226065180615,2.85606797749979)-- (-7.902113032590307,2.238033988749895);
\draw (-7.902113032590307,2.238033988749895)-- (-7.175570504584947,0.001966011250105315);
\draw (-7.175570504584947,0.001966011250105315)-- (-4.824429495415054,0.001966011250105093);
\draw (-4.824429495415054,0.001966011250105093)-- (-4.097886967409693,2.2380339887498946);
\draw (-4.097886967409693,2.2380339887498946)-- (-5.999999999999999,3.62);
\draw (-2.195773934819386,2.85606797749979)-- (-4.097886967409693,2.238033988749895);
\draw (-4.824429495415054,0.001966011250105093)-- (-4.097886967409693,2.2380339887498946);
\draw (-6.000000000000001,-5.6160679774997915)-- (-6.000000000000001,-3.6160679774997897);
\draw (-6.000000000000001,-3.6160679774997897)-- (-7.902113032590307,-2.2341019662496855);
\draw (-9.804226065180615,-2.8521359549995795)-- (-7.902113032590307,-2.2341019662496855);
\draw (-7.902113032590307,-2.2341019662496855)-- (-7.175570504584947,0.0019660112501051485);
\draw (-7.175570504584947,0.0019660112501051485)-- (-4.824429495415054,0.001966011250104982);
\draw (-4.824429495415054,0.001966011250104982)-- (-4.097886967409693,-2.234101966249685);
\draw (-4.097886967409693,-2.234101966249685)-- (-6.0,-3.6160679774997897);
\draw (-2.1957739348193863,-2.8521359549995813)-- (-4.097886967409694,-2.2341019662496855);
\draw (-4.824429495415054,0.001966011250104982)-- (-4.097886967409693,-2.234101966249685);
%\draw (2.195773934819386,2.85606797749979)-- (2.1957739348193854,-2.8521359549995795);
%\draw (6.0,5.62)-- (5.999999999999999,-5.6160679774997915);
%\draw (9.804226065180615,2.85606797749979)-- (9.804226065180615,-2.8521359549995813);
\draw (2.195773934819386,2.85606797749979)-- (6,2);
\draw (6,-2)-- (2.1957739348193854,-2.8521359549995795);
\draw (9.804226065180615,2.85606797749979)-- (6,2);
\draw (6,-2)-- (9.804226065180615,-2.8521359549995813);
\draw (6.0,5.62)-- (6,2);
\draw (6.0,-2)-- (5.999999999999999,-5.6160679774997915);
\draw (-5.9,5.8) node[anchor=north west] {$y_2$};
\draw (6.1,5.8) node[anchor=north west] {$y_2$};
\draw (-3,3.6) node[anchor=north west] {$y_3$};
\draw (-3,-2) node[anchor=north west] {$y_5$};
\draw (-5.9,-5.1) node[anchor=north west] {$y_6$};
\draw (-9.7,-2) node[anchor=north west] {$y_7$};
\draw (-9.7,3.6) node[anchor=north west] {$y_1$};
\draw (9,3.6) node[anchor=north west] {$y_3$};
\draw (2.3,3.6) node[anchor=north west] {$y_1$};
\draw (9,-2) node[anchor=north west] {$y_5$};
\draw (6.1, -5.1) node[anchor=north west] {$y_6$};
\draw (2.3,-2) node[anchor=north west] {$y_7$};
\draw (-6.4,3.5) node[anchor=north west] {$x_2$};
\draw (-5,2.5) node[anchor=north west] {$x_3$};
\draw (-5.6,.6) node[anchor=north west] {$x_4$};
\draw (-7.3, .6) node[anchor=north west] {$x_0$};
\draw (-7.8,2.5) node[anchor=north west] {$x_1$};
\draw (-7.8, -1.8) node[anchor=north west] {$x_7$};
\draw (-6.4,-2.8) node[anchor=north west] {$x_6$};
\draw (-5,-1.8) node[anchor=north west] {$x_5$};
\draw (5.7,1.9) node[anchor=north west] {$u$};
\draw (5.7,-1.3) node[anchor=north west] {$v$};
\begin{scriptsize}
\draw [fill=black] (-6.0,3.62) circle (1.5pt);
\draw [fill=black] (-7.902113032590307,2.238033988749895) circle (1.5pt);
\draw [fill=black] (-7.175570504584947,0.001966011250105315) circle (1.5pt);
\draw [fill=black] (-4.824429495415054,0.001966011250105093) circle (1.5pt);
\draw [fill=black] (-4.097886967409693,2.2380339887498946) circle (1.5pt);
\draw [fill=black] (-6.0,5.62) circle (1.5pt);
\draw [fill=black] (6.0,2) circle (1.5pt);
\draw [fill=black] (6.0,-2) circle (1.5pt);
\draw [fill=black] (-9.804226065180615,2.85606797749979) circle (1.5pt);
\draw [fill=black] (-7.902113032590307,2.238033988749895) circle (1.5pt);
\draw [fill=black] (-7.902113032590307,2.238033988749895) circle (1.5pt);
\draw [fill=black] (-7.175570504584947,0.001966011250105315) circle (1.5pt);
\draw [fill=black] (-7.175570504584947,0.001966011250105315) circle (1.5pt);
\draw [fill=black] (-7.175570504584947,0.001966011250105315) circle (1.5pt);
\draw [fill=black] (-4.824429495415054,0.001966011250105093) circle (1.5pt);
\draw [fill=black] (-7.175570504584947,0.001966011250105315) circle (1.5pt);
\draw [fill=black] (-4.824429495415054,0.001966011250105093) circle (1.5pt);
\draw [fill=black] (-4.824429495415054,0.001966011250105093) circle (1.5pt);
\draw [fill=black] (-4.824429495415054,0.001966011250105093) circle (1.5pt);
\draw [fill=black] (-4.097886967409693,2.2380339887498946) circle (1.5pt);
\draw [fill=black] (-4.824429495415054,0.001966011250105093) circle (1.5pt);
\draw [fill=black] (-4.097886967409693,2.2380339887498946) circle (1.5pt);
\draw [fill=black] (-4.824429495415054,0.001966011250105093) circle (1.5pt);
\draw [fill=black] (-4.824429495415054,0.001966011250105093) circle (1.5pt);
\draw [fill=black] (-4.097886967409693,2.2380339887498946) circle (1.5pt);
\draw [fill=black] (-4.097886967409693,2.2380339887498946) circle (1.5pt);
\draw [fill=black] (-4.097886967409693,2.2380339887498946) circle (1.5pt);
\draw [fill=black] (-5.999999999999999,3.62) circle (1.5pt);
\draw [fill=black] (-4.097886967409693,2.2380339887498946) circle (1.5pt);
\draw [fill=black] (-5.999999999999999,3.62) circle (1.5pt);
\draw [fill=black] (-4.097886967409693,2.2380339887498946) circle (1.5pt);
\draw [fill=black] (-4.097886967409693,2.2380339887498946) circle (1.5pt);
\draw [fill=black] (-4.097886967409693,2.2380339887498946) circle (1.5pt);
\draw [fill=black] (-5.999999999999999,3.62) circle (1.5pt);
\draw [fill=black] (-4.097886967409693,2.2380339887498946) circle (1.5pt);
\draw [fill=black] (-4.097886967409693,2.2380339887498946) circle (1.5pt);
\draw [fill=black] (-5.999999999999999,3.62) circle (1.5pt);
\draw [fill=black] (-2.195773934819386,2.85606797749979) circle (1.5pt);
\draw [fill=black] (-4.097886967409693,2.238033988749895) circle (1.5pt);
\draw [fill=black] (-4.824429495415054,0.001966011250105093) circle (1.5pt);
\draw [fill=black] (-4.097886967409693,2.2380339887498946) circle (1.5pt);
\draw [fill=black] (-4.824429495415054,0.001966011250105093) circle (1.5pt);
\draw [fill=black] (-4.824429495415054,0.001966011250105093) circle (1.5pt);
\draw [fill=black] (-4.824429495415054,0.001966011250105093) circle (1.5pt);
\draw [fill=black] (-4.824429495415054,0.001966011250105093) circle (1.5pt);
\draw [fill=black] (-4.824429495415054,0.001966011250105093) circle (1.5pt);
\draw [fill=black] (-4.824429495415054,0.001966011250105093) circle (1.5pt);
\draw [fill=black] (-4.097886967409693,2.2380339887498946) circle (1.5pt);
\draw [fill=black] (-4.824429495415054,0.001966011250105093) circle (1.5pt);
\draw [fill=black] (-4.824429495415054,0.001966011250105093) circle (1.5pt);
\draw [fill=black] (-4.824429495415054,0.001966011250105093) circle (1.5pt);
\draw [fill=black] (-4.097886967409693,2.2380339887498946) circle (1.5pt);
\draw [fill=black] (-4.824429495415054,0.001966011250105093) circle (1.5pt);
\draw [fill=black] (-4.824429495415054,0.001966011250105093) circle (1.5pt);
\draw [fill=black] (-4.097886967409693,2.2380339887498946) circle (1.5pt);
\draw [fill=black] (-6.000000000000001,-5.6160679774997915) circle (1.5pt);
\draw [fill=black] (-6.000000000000001,-3.6160679774997897) circle (1.5pt);
\draw [fill=black] (-6.000000000000001,-3.6160679774997897) circle (1.5pt);
\draw [fill=black] (-7.902113032590307,-2.2341019662496855) circle (1.5pt);
\draw [fill=black] (-9.804226065180615,-2.8521359549995795) circle (1.5pt);
\draw [fill=black] (-7.902113032590307,-2.2341019662496855) circle (1.5pt);
\draw [fill=black] (-7.902113032590307,-2.2341019662496855) circle (1.5pt);
\draw [fill=black] (-7.175570504584947,0.0019660112501051485) circle (1.5pt);
\draw [fill=black] (-7.175570504584947,0.0019660112501051485) circle (1.5pt);
\draw [fill=black] (-4.824429495415054,0.001966011250104982) circle (1.5pt);
\draw [fill=black] (-4.824429495415054,0.001966011250104982) circle (1.5pt);
\draw [fill=black] (-4.097886967409693,-2.234101966249685) circle (1.5pt);
\draw [fill=black] (-4.097886967409693,-2.234101966249685) circle (1.5pt);
\draw [fill=black] (-6.0,-3.6160679774997897) circle (1.5pt);
\draw [fill=black] (-2.1957739348193863,-2.8521359549995813) circle (1.5pt);
\draw [fill=black] (-4.097886967409694,-2.2341019662496855) circle (1.5pt);
\draw [fill=black] (-4.824429495415054,0.001966011250104982) circle (1.5pt);
\draw [fill=black] (-4.097886967409693,-2.234101966249685) circle (1.5pt);
\draw [fill=black] (-7.175570504584947,0.0019660112501051485) circle (1.5pt);
\draw [fill=black] (-4.824429495415054,0.001966011250104982) circle (1.5pt);
\draw [fill=black] (-4.097886967409693,-2.234101966249685) circle (1.5pt);
\draw [fill=black] (-7.175570504584947,0.0019660112501051485) circle (1.5pt);
\draw [fill=black] (-7.175570504584947,0.0019660112501051485) circle (1.5pt);
\draw [fill=black] (-4.824429495415054,0.001966011250104982) circle (1.5pt);
\draw [fill=black] (-4.824429495415054,0.001966011250104982) circle (1.5pt);
\draw [fill=black] (-4.824429495415054,0.001966011250104982) circle (1.5pt);
\draw [fill=black] (-4.097886967409693,-2.234101966249685) circle (1.5pt);
\draw [fill=black] (-4.824429495415054,0.001966011250104982) circle (1.5pt);
\draw [fill=black] (-4.824429495415054,0.001966011250104982) circle (1.5pt);
\draw [fill=black] (-4.097886967409693,-2.234101966249685) circle (1.5pt);
\draw [fill=black] (-4.097886967409693,-2.234101966249685) circle (1.5pt);
\draw [fill=black] (-4.097886967409693,-2.234101966249685) circle (1.5pt);
\draw [fill=black] (-6.0,-3.6160679774997897) circle (1.5pt);
\draw [fill=black] (-4.097886967409693,-2.234101966249685) circle (1.5pt);
\draw [fill=black] (-4.097886967409693,-2.234101966249685) circle (1.5pt);
\draw [fill=black] (-4.097886967409693,-2.234101966249685) circle (1.5pt);
\draw [fill=black] (-6.0,-3.6160679774997897) circle (1.5pt);
\draw [fill=black] (-4.097886967409693,-2.234101966249685) circle (1.5pt);
\draw [fill=black] (-4.097886967409693,-2.234101966249685) circle (1.5pt);
\draw [fill=black] (-6.0,-3.6160679774997897) circle (1.5pt);
\draw [fill=black] (-4.824429495415054,0.001966011250104982) circle (1.5pt);
\draw [fill=black] (-4.824429495415054,0.001966011250104982) circle (1.5pt);
\draw [fill=black] (-4.824429495415054,0.001966011250104982) circle (1.5pt);
\draw [fill=black] (-4.824429495415054,0.001966011250104982) circle (1.5pt);
\draw [fill=black] (-4.824429495415054,0.001966011250104982) circle (1.5pt);
\draw [fill=black] (-4.824429495415054,0.001966011250104982) circle (1.5pt);
\draw [fill=black] (-4.097886967409693,-2.234101966249685) circle (1.5pt);
\draw [fill=black] (-4.824429495415054,0.001966011250104982) circle (1.5pt);
\draw [fill=black] (-4.824429495415054,0.001966011250104982) circle (1.5pt);
\draw [fill=black] (-4.824429495415054,0.001966011250104982) circle (1.5pt);
\draw [fill=black] (-4.097886967409693,-2.234101966249685) circle (1.5pt);
\draw [fill=black] (-4.824429495415054,0.001966011250104982) circle (1.5pt);
\draw [fill=black] (-4.824429495415054,0.001966011250104982) circle (1.5pt);
\draw [fill=black] (-4.097886967409693,-2.234101966249685) circle (1.5pt);
\draw [fill=black] (2.1957739348193854,-2.8521359549995795) circle (1.5pt);
\draw [fill=black] (5.999999999999999,-5.6160679774997915) circle (1.5pt);
\draw [fill=black] (9.804226065180615,-2.8521359549995813) circle (1.5pt);
\draw [fill=black] (9.804226065180615,2.85606797749979) circle (1.5pt);
\draw [fill=black] (6.0,5.62) circle (1.5pt);
\draw [fill=black] (2.195773934819386,2.85606797749979) circle (1.5pt);
\draw [fill=black] (2.195773934819386,2.85606797749979) circle (1.5pt);
\draw [fill=black] (2.1957739348193854,-2.8521359549995795) circle (1.5pt);
\draw [fill=black] (6.0,5.62) circle (1.5pt);
\draw [fill=black] (5.999999999999999,-5.6160679774997915) circle (1.5pt);
\draw [fill=black] (9.804226065180615,2.85606797749979) circle (1.5pt);
\draw [fill=black] (9.804226065180615,-2.8521359549995813) circle (1.5pt);
\end{scriptsize}
\end{tikzpicture}
\caption{Forming $G'$ from $G$}
\label{fig:No5-5face}
\end{figure}

\begin{lemma}\label{No5-5face}
No two 5-faces in $G$ share an edge.
\end{lemma}
\begin{proof}
Suppose the contrary.  By Lemma \ref{No2on5cycle}, the boundaries of the two faces form an 8-cycle, $x_0x_1\dots x_7$ with $x_4x_0 \in E(G)$.  By Lemmas \ref{NoTriangle}, \ref{No4cycle}, \ref{separating} and \ref{No2on5cycle}, each $x_i$ other than $x_4,x_0$ has a third neighbor $y_i$ not on the 8-cycle that are distinct from each other, except possibly $y_2 = y_6$.  Additionally, the only possible adjacencies between the $y_i$'s are $y_iy_j$ for $i \in [3]$ and $j \in \{5,6,7\}$.  

Let $G'$ denote the graph obtained from $G$ by removing $x_0,\dots, x_7$, adding two new vertices $u,v$ and the edges $uy_1,uy_2, uy_3, vy_5, vy_6, vy_7$ (see Figure \ref{fig:No5-5face}).  Observe that $G'$ is a subcubic, planar multigraph, and so by the minimality of $G$, $G'$ has a good coloring, which ignoring $uy_1,uy_2, uy_3, vy_5, vy_6, vy_7$ gives us a good partial coloring of $G$ that can be extended by coloring $x_jy_j$ with the same color as $uy_j, j \in [3]$ and $x_\ell y_\ell$ with the same color as $vy_\ell$, for $\ell \in \{5,6,7\}$.  This new partial coloring of $G$ is still a good partial coloring, and we will refer to it as $\phi$.  

By the construction of $G'$, we see that $\phi(x_1y_1) \neq \phi(x_3y_3)$ and $\phi(x_5y_5) \neq \phi(x_7y_7)$.  Without loss of generality, we may assume that $\phi(x_1y_1) = 1$ and $\phi(x_3y_3) = 2$.  We will break the following into cases depending on $(\phi(x_5y_5), \phi(x_7y_7))$.  

\setcounter{case}{0}
\begin{case}
$(\phi(x_5y_5), \phi(x_7y_7)) = (3,4)$.
\end{case}

Observe that $|A_\phi(x_ix_{i+1})| \ge 2$ for $i \in \{1,2,5,6\}$, $|A_\phi(x_jx_{j+1})| \ge 4$ for $j \in \{0,3,4,7\}$ taken modulo 8 and $A_\phi(x_4x_0) = \{5,6,7,8,9\}$.  By the construction of $G'$, we can extend $\phi$ to another good partial coloring of $G$ by coloring $x_3x_4, x_4x_5, x_7x_0, x_0x_1$ with $1,4,3,2$, respectively.  We will call this good partial coloring $\sigma$.  Note that $|A_\sigma(x_ix_{i+1})| \ge 1$ for $i \in \{1,2,5,6\}$ and $A_\sigma(x_4x_0) = \{5,6,7,8,9\}$.  

If $|A_\sigma(x_1x_2) \cup A_\sigma(x_2x_3)|, |A_\sigma(x_5x_6) \cup A_\sigma(x_6x_7)| \ge 2$, we can color $x_1x_2$,$x_2x_3$, $x_5x_6$, $x_6x_7, x_4x_0$ in this order to obtain a good coloring of $G$.   By symmetry, we have two subcases to consider. 

\begin{subcase}%
$|A_\sigma(x_{1}x_{2}) \cup A_\sigma(x_{2}x_{3})| = |A_\sigma(x_{5}x_{6}) \cup A_\sigma(x_{6}x_{7})| = 1$. 
\end{subcase}

Let $A_\sigma(x_{1}x_{2}) = A_\sigma(x_{2}x_{3}) = \{\alpha\}$ and $A_\sigma(x_{5}x_{6}) = A_\sigma(x_{6}x_{7}) = \{\beta\}$.  Since $\alpha \notin \{1,2,3,4\}$, we  have $3 \in \Upsilon_\sigma(y_2,x_2) \cup \Upsilon(y_3,x_3)$.  However, if $3 \in \sU_\sigma(y_2)$, then $|A_\sigma(x_1x_2)| \ge 2$, a contradiction.  Thus, $\sU_\sigma(y_3) = \{2,3,\gamma\}$ for some $\gamma \notin [4]$, since $G$ is a counterexample.  By a similar argument, we  have $4 \in \sU_\sigma(y_1)$, and as $A_\sigma(x_1x_2) = A_\sigma(x_2x_3)$, we have $\sU_\sigma(y_1) = \{1,4,\gamma\}$.  Symmetrically, $\sU_\sigma(y_5) = \{2,3,\delta\}$ and $\sU_\sigma(y_7) = \{1,4,\delta\}$, where $\delta \notin [4]$.

Now, as $4 \in \sU_\sigma(y_1)$ and $|A_\sigma(x_1x_2)| = 1$, we cannot have $4 \in \sU_\sigma(y_2)$.  Thus, $4 \in A_\phi(x_2x_3)$.  Similarly, $2 \in A_\phi(x_6x_7)$.  Thus, we can extend $\phi$ by coloring $x_1x_2$ with $\alpha$, $x_2x_3$ with 4, $x_3x_4$ with 1, $x_5x_6$ with $\beta$, $x_6x_7$ with 2, $x_7x_0$ with 3 and color $x_4x_5, x_0x_1, x_4x_0$ in this order.  This gives us a good partial coloring of $G$ and completes this subcase.

\begin{subcase}%
$|A_\sigma(x_{1}x_{2}) \cup A_\sigma(x_{2}x_{3})| \geq 2$ and $|A_\sigma(x_{5}x_{6}) \cup A_\sigma(x_{6}x_{7})| = 1$. 
\end{subcase}

Suppose  $A_\sigma(x_{5}x_{6}) = A_\sigma(x_{6}x_{7}) = \{\beta\}$.   Now $2 \notin \sU_\phi(y_6) \cup \sU_\phi(y_7)$, as otherwise $|A_\sigma(x_6x_7)| \ge 2$, a contradiction.  Thus, $2 \in A_\phi(x_6x_7)$, and by symmetry, $1 \in A_\phi(x_5x_6)$.  Now, we can alter $\sigma$ to another good partial coloring by uncoloring $x_0x_1$ and then coloring $x_5x_6$ with $\beta$ and $x_6x_7$ with 2.  Call this new partial coloring $\psi$.  Note that $|A_{\psi}(x_0x_1)| \ge 2$ and $|A_{\psi}(x_ix_{i+1})| \ge1$ for $i \in \{1,2\}$.  Since the only change affecting the edges available on $x_1x_2, x_2x_3$ was the uncoloring of $x_0x_1$, we still have $|A_{\psi}(x_1x_2) \cup A_{\psi}(x_2x_3)| \ge 2$.  

If $|A_{\psi}(x_0x_1) \cup A_{\psi}(x_1x_2) \cup A_{\psi}(x_2x_3)| \ge 3$, then we can obtain a good coloring of $G$ by SDR.  So we  have $|A_{\psi}(x_0x_1)| = 2$ and $A_{\psi}(x_1x_2) \cup A_{\psi}(x_2x_3) = A_{\psi}(x_0x_1)$.  In particular, $A_\psi(x_1x_2) \subseteq A_\psi(x_0x_1)$.

Since $|A_{\psi}(x_0x_1)| = 2$ and $x_0x_1$ sees $x_7y_7$ colored 4, we cannot have $4 \in \sU_{\psi}(y_1) \cup \{\psi(x_2y_2)\}$.  If $4 \notin \Upsilon_\psi(y_2,x_2)$, then $4 \in A_\psi(x_1x_2)\setminus A_\psi(x_0x_1)$, a contradiction to $A_\psi(x_1x_2)\subseteq A_\psi(x_0x_1)$.  Thus, $4 \in \Upsilon_\psi(y_2,x_2)$, and so $|A_{\psi}(x_2x_3)| = 2$.  Furthermore, we cannot have 4  in $\sU_{\psi}(y_3) = \sU_\sigma(y_3)$, as otherwise $|A_\psi(x_2x_3)| \ge 3$.  Returning to $\phi$, this implies $4 \in A_\phi(x_3x_4)$. 

Recall that $1 \in A_\phi(x_5x_6)$.  By a symmetric argument, $3 \in \Upsilon_\psi(y_2,x_2)$.  Thus $\Upsilon_\psi(y_2,x_2) = \Upsilon_\sigma(y_2,x_2) = \{3,4\}$.  Now, we can alter $\sigma$ by first uncoloring $x_4x_5$, then recoloring $x_3x_4$ with 4 and coloring $x_5x_6$ with 1, $x_6x_7$ with $\beta$.  By the above, this is another good partial coloring, call it $\tau$.  

Note that $|A_{\tau}(x_4x_5)| \ge 1, |A_{\tau}(x_1x_2)|, |A_{\tau}(x_2x_3)| \ge 2$ and $|A_{\tau}(x_4x_0)| \ge 4$.  We can then color $x_4x_5, x_2x_3, x_1x_2, x_4x_0$ in this order to obtain a good coloring of $G$.

This completes the subcase and so proves the case.

\begin{case}
$(\phi(x_5y_5), \phi(x_7y_7)) = (1,3)$.
\end{case}

First, notice that one can recolor $x_1y_1$ with a color other than 1, call it $\alpha$, and still maintain a good partial coloring of $G$.  We will proceed in this case based on whether or not $\alpha$ is 2.

\begin{subcase}
$\alpha \neq 2$.
\end{subcase}

We can extend our good partial coloring of $G$ by coloring $x_2x_3, x_7x_0$ with 1, $x_4x_5$ with 3 and $x_0x_1$ with 2.  Call this new coloring $\sigma$.  

Note that $|A_\sigma(x_1x_2)|, |A_\sigma(x_6x_7)| \ge 1$, $|A_\sigma(x_5x_6)| \ge 2$, $|A_\sigma(x_3x_4)| \ge 3$ and $|A_\sigma(x_4x_0)| \ge 5$.  Thus, we can color $x_6x_7, x_5x_6, x_1x_2, x_3x_4, x_4x_0$ in this order to obtain a good coloring of $G$.

\begin{subcase}\label{(1,3).2}
$\alpha = 2$.
\end{subcase}

We can extend our good partial coloring of $G$ by coloring $x_2x_3, x_7x_0$ with 1 and $x_4x_5$ with 3.  Call this new coloring $\sigma$.  

Note that $|A_\sigma(x_ix_{i+1})| \ge 2$ for $i \in \{1,5,6\}$, $|A_\sigma(x_3x_4)|, |A_\sigma(x_0x_1)| \ge 3$ and $|A_\sigma(x_4x_0)| \ge 6$.  If there exists some $\beta \in A_\sigma(x_6x_7) \cap A_\sigma(x_1x_2)$,  we can color $x_1x_2, x_6x_7$ with $\beta$ and then color $x_5x_6, x_3x_4, x_1x_0, x_4x_0$ in this order to obtain a good partial coloring of $G$.  

As a result, either $|A_\sigma(x_1x_0)| \ge 4$ or there exists some $\gamma \in (A_\sigma(x_1x_2) \cup A_\sigma(x_6x_7)) \setminus A_\sigma(x_1x_0)$.  In either case, we color $x_1x_2, x_6x_7$ in this order (in particular, using $\gamma$ on at least one edge in the latter case), then color $x_5x_6, x_3x_4, x_1x_0, x_4x_0$ in this order to obtain a good coloring of $G$.

This completes the subcase, and so proves the case.

\begin{case}
$(\phi(x_5y_5), \phi(x_7y_7)) = (1,2)$.
\end{case}

As in the previous case, we can recolor $x_1y_1$ with a color $\alpha \neq 1$ so that we still maintain a good partial coloring of $G$.  We proceed in subcases as above.

\begin{subcase}\label{(1,2).1}
$\alpha = 2$.
\end{subcase}

We can extend our good partial coloring of $G$ by coloring $x_2x_3, x_7x_0$ with 1.  Call this new coloring $\sigma$.  

Note that $|A_\sigma(x_ix_{i+1})| \ge 2$ for $i \in \{1,5,6\}$, $|A_\sigma(x_jx_{j+1})| \ge 4$ for $j \in \{0,3,4\}$ modulo 8 and $|A_\sigma(x_4x_0)| \ge 7$.  Now, either $|A_\sigma(x_1x_2)| \ge 4$ or there exists $\beta \in A_\sigma(x_3x_4) \setminus A_\sigma(x_1x_2)$.  In either case, we color $x_3x_4$ first (in particular, with $\beta$ in the latter case), then color $x_5x_6, x_6x_7, x_4x_5, x_0x_1, x_1x_2, x_4x_0$ to obtain our good coloring of $G$.

\begin{subcase}
$\alpha \neq 2$.
\end{subcase}

Just as with $x_1y_1$, we can recolor $x_3y_3$ with another color $\beta \neq 2$ and still maintain a good partial coloring of $G$.  By the above subcase, we may assume that $\beta \neq 1$, but it is possible that $\alpha = \beta$.  We can  extend our good partial coloring of $G$ by coloring $x_1x_2, x_4x_5$ with 2 and $x_2x_3, x_7x_0$ with 1.  Call this new coloring $\sigma$.

Note that $|A_\sigma(x_5x_6)|, |A_\sigma(x_6x_7)| \ge 2$, $|A_\sigma(x_3x_4)|, |A_\sigma(x_0x_1)| \ge 3$ and $|A_\sigma(x_4x_0)| \ge 5$.  We can then color $x_5x_6, x_6x_7, x_0x_1, x_3x_4, x_4x_0$ in this order to obtain a good partial coloring of $G$.

This completes the subcase and so proves the case.

\begin{case}
$(\phi(x_5y_5), \phi(x_7y_7)) = (2,1)$.
\end{case}

Again, we recolor $x_1y_1$ with $\alpha \neq 1$.  

\begin{subcase}
$\alpha = 2$.
\end{subcase}

This subcase is symmetric to Subcase \ref{(1,2).1}.

\begin{subcase}
$\alpha \neq 2$.
\end{subcase}

We can extend our good partial coloring of $G$ by coloring $x_1x_2, x_4x_5$ with 1 and $x_7x_0$ with 2.  Call this new coloring $\sigma$.

Note that $|A_\sigma(x_2x_3)| \ge 1, |A_\sigma(x_5x_6)|, |A_\sigma(x_6x_7)| \ge 2, |A_\sigma(x_0x_1)| \ge 3$, $|A_\sigma(x_3x_4)| \ge 4$ and $|A_\sigma(x_4x_0)| \ge 6$.  We can color $x_2x_3, x_5x_6, x_6x_7, x_0x_1, x_3x_4, x_4x_0$ in this order to obtain a good partial coloring of $G$.  This completes the subcase and so completes the case.

\begin{case}
$(\phi(x_5y_5), \phi(x_7y_7)) = (3,1)$.
\end{case}

Observe that $|A_\phi(x_ix_{i+1})| \ge 2$ for $i \in \{1,2,5,6\}$, $|A_\phi(x_jx_{j+1})| \ge 4$ for $j \in \{3,4\}$, $|A_\phi(x_\ell x_{\ell+1})| \ge 5$ for $\ell \in \{0,7\}$ modulo 8 and $A_\phi(x_4x_0) = \{4,5,6,7,8,9\}$.  We can extend $\phi$ by coloring $x_3x_4, x_7x_0, x_0x_1$ with 1,3,2 respectively.  We can further extend this new coloring by coloring $x_1x_2, x_2x_3$ in this order as $|A_\phi(x_1x_2)\setminus \{1,2,3\}| \ge 1$ and $|A_\phi(x_2x_3) \setminus\{1,2,3\}| \ge 2$.  This is another good partial coloring of $G$, and we will refer to it as $\sigma$ in this case.  Let $\alpha := \sigma(x_2x_3)$, and since $x_1x_2$  sees $1,2,3$, we may assume that $\sigma(x_1x_2) = 4$

Note that $|A_\sigma(x_6x_7)| \ge 1$, $|A_\sigma(x_5x_6)|, |A_\sigma(x_4x_5)| \ge 2$ and $|A_\sigma(x_4x_0)| \ge 5$.  We  have $A_\sigma(x_6x_7) \subseteq A_\sigma(x_5x_6) = A_\sigma(x_4x_5)$ and $|A_\sigma(x_4x_5)| =2$, otherwise we obtain a good coloring of $G$ by SDR.   So let $A_\sigma(x_5x_6) = A_\sigma(x_4x_5) = \{\beta_1, \beta_2\}$.  Note that $1,2,3,\alpha \notin \{\beta_1,\beta_2\}$.

Since $|A_\sigma(x_4x_5)| = 2$ and $x_4x_5$ sees 2 and $\alpha$, we cannot have $2,\alpha \in \sU_\sigma(y_5) \cup \{\sigma(x_6y_6)\}$.  As $x_5x_6$ must also see 2 and $\alpha$, we  have $\Upsilon_\sigma(y_6,x_6) = \{2,\alpha\}$.  Thus, $|A_\sigma(x_6x_7)| \ge 2$, and in particular, $A_\sigma(x_6x_7) = \{\beta_1,\beta_2\}$ as $A_\sigma(x_6x_7) \subseteq A_\sigma(x_4x_5)$.

Now, we can return to $\phi$ and obtain a different partial coloring of $G$ by coloring $x_4x_5$ with 1, $x_5x_6$ with $\beta_1$, $x_6x_7$ with $\beta_2$, $x_7x_0$ with 3 and $x_0x_1$ with 2.  This partial coloring is also good, and we will denote it by $\psi_1$.  

Note that $|A_{\psi_1}(x_1x_2)| \ge 1$ and $|A_{\psi_1}(x_2x_3)|, |A_{\psi_1}(x_3x_4)| \ge 2$.  As above, we  have $A_{\psi_1}(x_1x_2)$ $\subseteq$ $A_{\psi_1}(x_2x_3) = A_{\psi_1}(x_3x_4)$ and $|A_{\psi_1}(x_2x_3)| = 2$, otherwise we obtain a good coloring of $G$ by SDR.  As $x_3x_4$ sees $3,\beta_1$ and $|A_{\psi_1}(x_3x_4)| = 2$, we cannot have $3,\beta_1 \in \sU_{\psi_1}(y_3) \cup \{\psi_1(x_2y_2)\}$.  However, as $A_{\psi_1}(x_2x_3) = A_{\psi_1}(x_3x_4)$, we  have $\Upsilon_{\psi_1}(y_2,x_2) = \{3,\beta_1\}$.  Note that $\Upsilon_\phi(y_2,x_2) = \{3,\beta_1\}$ as a result.

Now, if we switch $\beta_1, \beta_2$ so that $x_5x_6$ is colored with $\beta_2$ and $x_6x_7$ is colored with $\beta_1$, we still have a good partial coloring of $G$, call it $\psi_2$.  The same argument however, shows that $\Upsilon_{\psi_2}(y_2,x_2) = \{3,\beta_2\}$, so that $\Upsilon_\phi(y_2,x_2) = \{3,\beta_2\}$ and $\beta_1 = \beta_2$, a contradiction.  This completes the proof of the case.

As we have exhausted all cases, the lemma holds.
\end{proof}

\begin{figure}[h]
\centering
\begin{tikzpicture}[line cap=round,line join=round,>=triangle 45,x=0.75cm,y=0.75cm]
\clip(-11,-6) rectangle (11,6);
\draw (-6.0,-3.62)-- (-6.0,-5.62);
\draw (-4.097886967409693,-2.238033988749895)-- (-2.195773934819386,-2.85606797749979);
\draw (-7.902113032590307,-2.238033988749895)-- (-9.804226065180615,-2.85606797749979);
\draw (-3.652480562338399,2.032090915285201)-- (-1.6524805623383987,2.032090915285201);
\draw (-4.828051066923345,4.064056926535307)-- (-3.8280510669233445,5.796107734104184);
\draw (-7.175570504584946,4.061966011250106)-- (-8.175570504584945,5.794016818818983);
\draw (-8.347519437661601,2.027909084714799)-- (-10.347519437661601,2.0279090847148002);
\draw (-7.175570504584946,4.061966011250106)-- (-4.828051066923345,4.064056926535307);
\draw (-4.828051066923345,4.064056926535307)-- (-3.652480562338399,2.032090915285201);
\draw (-3.652480562338399,2.032090915285201)-- (-4.824429495415053,-0.001966011250105426);
\draw (-4.824429495415053,-0.001966011250105426)-- (-7.175570504584946,-0.001966011250105204);
\draw (-7.175570504584946,-0.001966011250105204)-- (-8.347519437661601,2.027909084714799);
\draw (-8.347519437661601,2.027909084714799)-- (-7.175570504584946,4.061966011250106);
\draw (-7.175570504584946,-0.001966011250105204)-- (-7.902113032590307,-2.2380339887498946);
\draw (-4.824429495415053,-0.001966011250105426)-- (-4.097886967409693,-2.238033988749895);
\draw (-7.902113032590307,-2.2380339887498946)-- (-6.0,-3.62);
\draw (-6.0,-3.62)-- (-4.097886967409693,-2.238033988749895);
\draw [->] (-1.0,0.0) -- (1.0,0.0);
\draw (-7.2,4) node[anchor=north west] {$u_2$};
\draw (-5.5,4) node[anchor=north west] {$u_3$};
\draw (-4.7,2.4) node[anchor=north west] {$u_4$};
\draw (-5.5,0.8) node[anchor=north west] {$u_5$};
\draw (-5.1,-1.7) node[anchor=north west] {$u_6$};
\draw (-6.4,-2.7) node[anchor=north west] {$u_7$};
\draw (-7.7,-1.7) node[anchor=north west] {$u_8$};
\draw (-7.3,0.8) node[anchor=north west] {$u_0$};
\draw (-8.1,2.4) node[anchor=north west] {$u_1$};
\draw (-8,5.8) node[anchor=north west] {$v_2$};
\draw (4,5.8) node[anchor=north west] {$v_2$};
\draw (-4.9,5.8) node[anchor=north west] {$v_3$};
\draw (3.824429495415055,5.794016818818983)-- (8.171948933076656,5.796107734104184);
\draw (1.6524805623383987,2.0279090847148002)-- (2.195773934819386,-2.85606797749979);
\draw (10.347519437661601,2.032090915285201)-- (9.804226065180615,-2.85606797749979);
\draw (7.1, 5.8) node[anchor=north west] {$v_3$};
\draw (-2.4,2.7) node[anchor=north west] {$v_4$};
\draw (-3,-2) node[anchor=north west] {$v_6$};
\draw (-5.767057190082635,-4.817012066115719) node[anchor=north west] {$v_7$};
\draw (-9.8,-2) node[anchor=north west] {$v_8$};
\draw (-10.4,2.7) node[anchor=north west] {$v_1$};
\draw (1.6,2.7) node[anchor=north west] {$v_1$};
\draw (2.2, -2) node[anchor=north west] {$v_8$};
\draw (9,-2) node[anchor=north west] {$v_6$};
\draw (9.6,2.7) node[anchor=north west] {$v_4$};
\draw (6.139342809917368,-4.599212066115719) node[anchor=north west] {$v_7$};
\begin{scriptsize}
\draw [fill=black] (-6.0,-3.62) circle (1.5pt);
\draw [fill=black] (-4.097886967409693,-2.238033988749895) circle (1.5pt);
\draw [fill=black] (-4.824429495415053,-0.001966011250105426) circle (1.5pt);
\draw [fill=black] (-7.175570504584946,-0.001966011250105204) circle (1.5pt);
\draw [fill=black] (-7.902113032590307,-2.2380339887498946) circle (1.5pt);
\draw [fill=black] (-3.652480562338399,2.032090915285201) circle (1.5pt);
\draw [fill=black] (-4.828051066923345,4.064056926535307) circle (1.5pt);
\draw [fill=black] (-7.175570504584946,4.061966011250106) circle (1.5pt);
\draw [fill=black] (-8.347519437661601,2.027909084714799) circle (1.5pt);
\draw [fill=black] (-6.0,-5.62) circle (1.5pt);
\draw [fill=black] (-4.097886967409693,-2.238033988749895) circle (1.5pt);
\draw [fill=black] (-2.195773934819386,-2.85606797749979) circle (1.5pt);
\draw [fill=black] (-7.902113032590307,-2.238033988749895) circle (1.5pt);
\draw [fill=black] (-9.804226065180615,-2.85606797749979) circle (1.5pt);
\draw [fill=black] (-1.6524805623383987,2.032090915285201) circle (1.5pt);
\draw [fill=black] (-4.828051066923345,4.064056926535307) circle (1.5pt);
\draw [fill=black] (-3.8280510669233445,5.796107734104184) circle (1.5pt);
\draw [fill=black] (-7.175570504584946,4.061966011250106) circle (1.5pt);
\draw [fill=black] (-8.175570504584945,5.794016818818983) circle (1.5pt);
\draw [fill=black] (-7.175570504584946,4.061966011250106) circle (1.5pt);
\draw [fill=black] (-8.347519437661601,2.027909084714799) circle (1.5pt);
\draw [fill=black] (-10.347519437661601,2.0279090847148002) circle (1.5pt);
\draw [fill=black] (-8.347519437661601,2.027909084714799) circle (1.5pt);
\draw [fill=black] (-8.347519437661601,2.027909084714799) circle (1.5pt);
\draw [fill=black] (3.824429495415055,5.794016818818983) circle (1.5pt);
\draw [fill=black] (8.171948933076656,5.796107734104184) circle (1.5pt);
\draw [fill=black] (10.347519437661601,2.032090915285201) circle (1.5pt);
\draw [fill=black] (9.804226065180615,-2.85606797749979) circle (1.5pt);
\draw [fill=black] (6.0,-5.62) circle (1.5pt);
\draw [fill=black] (2.195773934819386,-2.85606797749979) circle (1.5pt);
\draw [fill=black] (1.6524805623383987,2.0279090847148002) circle (1.5pt);
\end{scriptsize}
\end{tikzpicture}
\caption{Forming $G'$ from $G$}
\label{fig:No5-6face}
\end{figure}

\begin{lemma}\label{No5-6face}
No 5-face in $G$ can share an edge with a 6-face. 
\end{lemma}

\begin{proof}
Suppose that a 5-face and a 6-face share an edge.  By Lemmas \ref{NoTriangle} and \ref{No2on5cycle}, their boundaries form a 9-cycle, $u_0 u_1 \dots u_8$ so that $u_5u_0 \in E(G)$ .  By Lemmas \ref{No2on5cycle} and \ref{No2on6cycle}, each $u_i$ is a 3-vertex.   Additionally, Lemmas \ref{NoTriangle},  \ref{separating} and \ref{No4cycle}  imply that each $u_i$ other than $u_5, u_0$ has a third neighbor $v_i$ not on the 9-cycle.  By these same lemmas, the vertices $v_1, v_2, v_3, v_4, u_6, u_8$ are distinct from each other, as are the vertices $u_4, u_1, v_6, v_7, v_8$.  

By Lemmas \ref{separating}, \ref{No4cycle},  and \ref{No5-5face}, the edges $v_2v_3, v_4v_6, v_8v_1$ do not exist.  So let $G'$ denote the graph obtained from $G$ by deleting $u_1, u_2, \dots, u_0$ and adding the edges $v_2v_3, v_4v_6, v_8v_1$ (see Figure \ref{fig:No5-6face}). Observe that $G'$ is a subcubic, planar multigraph, and so by the minimality of $G$, $G'$ has a good coloring.  Ignoring $v_2v_3,v_4v_6, v_8v_1$, we have a good partial coloring of $G$ that we can extend by coloring $u_1v_1, u_8v_8$ with the same color that $v_8v_1$ received in $G'$ and $u_4v_4, u_6v_6$ with the same color that $v_4v_6$ received in $G'$.  We can further extend this good partial coloring of $G$ by coloring $u_2v_2, u_3v_3$ and $u_7v_7$.  Call this extended, good partial coloring, $\phi$, and let $\alpha$ denote $\phi(u_7v_7)$.

\setcounter{case}{0}
\begin{case}\label{Case1}
$\phi(u_1v_1) \neq \phi(u_4v_4)$.
\end{case}

Without loss of generality, we may assume that $\phi(u_1v_1) = \phi(u_8v_8) = 2$ and $\phi(u_4v_4) = \phi(u_6v_6) = 1$.  

\begin{subcase}\label{1.1}
$1 \in \Upsilon_\phi(v_1,u_1)$ and $2 \in \Upsilon_\phi(v_4,u_4)$.
\end{subcase}

By the existence of $v_4v_6, v_8v_1$ in our auxiliary graph $G'$, we cannot have $2 \in \sU_\phi(v_6)$ or $1 \in \sU_\phi(v_8)$.  So, we can extend $\phi$ to another good partial coloring of $G$ by coloring $u_5u_6$ with 2 and $u_8u_0$ with 1.  Call this new coloring $\sigma$.  

Observe $|A_\sigma(u_2u_3)| \ge 1$, $|A_\sigma(u_iu_{i+1})| \ge 2$ for $i \in \{1,3,6,7\}$, $|A_\sigma(u_4u_5)|, |A_\sigma(u_0u_1)| \ge 5$ and $|A_\sigma(u_5u_0)| \ge 7$.  Thus, if we can somehow extend $\sigma$ to a good partial coloring on $u_1u_2, u_2u_3, u_3u_4$, we can further extend this to a good coloring of $G$ by coloring $u_6u_7, u_7u_8$, $u_4u_5$, $u_0u_1$, $u_5u_0$ in this order.  Thus, it suffices to color $u_1u_2, u_2u_3, u_3u_4$.

If we cannot, then we  have $A_\sigma(u_1u_2) = A_\sigma(u_3u_4)$ and $|A_\sigma(u_1u_2)| = 2$.  As $1,2 \notin A_\sigma(u_1u_2)$, we may assume that $A_\sigma(u_1u_2) = A_\sigma(u_3u_4) = \{8,9\}$.  Additionally, we may assume that $\sU_\sigma(v_4) = \{1,2, 3\}, A_\sigma(v_3) = \{4,5,6\}$ with $\sigma(u_3v_3) = 4$ and $\sigma(u_2v_2) = 7$.  Since $A_\sigma(u_1u_2) = A_\sigma(u_3u_4)$, we have 5 or 6 in $\Upsilon_\sigma(v_2,u_2)$.  However, $v_2v_3$ is an edge in our auxiliary graph $G'$ so that $\Upsilon_\sigma(v_2,u_2) \cap \Upsilon_\sigma(v_3,u_3) = \emptyset$, a contradiction.

\begin{subcase}\label{1.2}
$1 \in \Upsilon_\phi(v_1,u_1)$, but $2 \notin \Upsilon(v_4,u_4)$.
\end{subcase}

Recall that $\phi$ colors both $u_2v_2$ and $u_3v_3$.  In this case, we may choose $\phi(u_3v_3)$ so that $\phi(u_3v_3) \neq 2$.   As a result, $2 \in A_\phi(u_4u_5)$.  As in Subcase \ref{1.1}, we can extend $\phi$ by coloring $u_8u_0$ with 1.  Call this new, good partial coloring $\sigma$.  We proceed to prove this subcase by considering whether or not 2 is in $\Upsilon_\sigma(v_3,u_3)$.

\begin{subsubcase}
$2 \notin \Upsilon_\sigma(v_3,u_3)$.
\end{subsubcase}

As a result, $2 \in A_\sigma(u_3u_4)$, and we can extend $\sigma$ by coloring $u_3u_4$ with 2, and then $u_2u_3, u_1u_2$ in this order.  Call this good partial coloring $\psi$.  Observe that $|A_\psi(u_6u_7)|, |A_\psi(u_7u_8)| \ge 2$, $|A_\psi(u_4u_5)|, |A_\psi(u_0u_1)| \ge 3$, $|A_\psi(u_5u_6)| \ge 4$ and $|A_\psi(u_5u_0)| \ge 6$.  If $|A_\psi(u_4u_5) \cup A_\psi(u_7u_8)| \ge 5$, then we obtain a good coloring of $G$ by SDR.  Otherwise, there exists some $\beta$ with which we can color $u_4u_5, u_7u_8$ and then color $u_6u_7, u_0u_1, u_5u_6, u_5u_0$ in this order to obtain a good coloring of $G$.

\begin{subsubcase}
$2 \in \Upsilon_\phi(v_3,u_3)$.
\end{subsubcase}

Recall that $2 \in A_\sigma(u_4u_5)$.    Additionally, we can recolor $u_1v_1$ with some $\beta \neq 2$ and still maintain a good partial coloring of $G$.  Thus, we adjust $\sigma$ by recoloring $u_1v_1$ with $\beta$, coloring $u_1u_2, u_4u_5$ with 2 and then coloring $u_2u_3, u_3u_4$ in this order.  Call this good partial coloring $\psi$.  

Observe that $|A_\psi(u_6u_7)|, |A_\psi(u_7u_8)| \ge 2$, $|A_\psi(u_5u_6)|, |A_\psi(u_0u_1)| \ge 3$ and $|A_\psi(u_5u_0)| \ge 5$.  We then color $u_6u_7, u_7u_8, u_5u_6, u_0u_1, u_5u_0$ in this order to obtain a good coloring of $G$.

This completes the subcase, and by symmetry, it remains to consider the following subcase.

\begin{subcase}
$1,2 \notin \Upsilon_\phi(v_1,u_1) \cup \Upsilon_\phi(v_4,u_4)$.
\end{subcase}

Just as in Subcase \ref{1.2}, we may assume that $\phi(u_3v_3) \neq 2$, and as a result, $2 \in A_\phi(u_4u_5)$. We proceed to prove this final subcase based on the color of $\phi(u_2v_2)$.
 
\begin{subsubcase}\label{1.3.1}
$\phi(u_2v_2) \neq 1$.
\end{subsubcase}

As a result, $1 \in A_\phi(u_0u_1)$.  Additionally, there exists some color in $A_\phi(u_2u_3)$.  Thus, we can extend $\phi$ to another good partial coloring of $G$ by coloring $u_1u_0$ with 1, $u_4u_5$ with 2 and then coloring $u_2u_3$ with some available color.  We can further extend $\phi$ by coloring $u_6u_7$ and $u_7u_8$ with some $\beta$ and $\gamma$, respectively.  Call this good partial coloring $\sigma$.

Now, we can choose $\beta$ and $\gamma$ such that either $\{\alpha,\beta\} \neq \Upsilon_\sigma(v_1,u_1)$ or $\{\alpha,\gamma\} \neq \Upsilon_\sigma(v_4,u_4)$.  We show the former as the latter is done by a similar argument.  Since $|A_\phi(u_6u_7)|$, $|A_\phi(u_7u_8)|$ $\ge 2$, if $\alpha \notin \Upsilon_\phi(v_1,u_1)$, then we are done, and if $\alpha \in \Upsilon_\phi(v_1,u_1)$, then we can choose $\beta$ from $A_\phi(u_6u_7)\setminus\Upsilon_\phi(v_1,u_1)$.  

Now, if $\Upsilon_\phi(v_1,u_1) \cap \Upsilon_\phi(v_4,u_4) = \emptyset$, then we can choose $\beta$ and $\gamma$ such that both $\Upsilon_\phi(v_1,u_1) \neq \{\alpha,\beta\}$ and $\Upsilon_\phi(v_4,u_4) \neq \{\alpha,\gamma\}$.  Indeed, if $\alpha \notin \Upsilon_\phi(v_1,u_1) \cup \Upsilon_\phi(v_4,u_4)$, then we are done.  So either $\alpha \in \Upsilon_\phi(v_1,u_1)\setminus\Upsilon_\phi(v_4,u_4)$ or $\alpha \in \Upsilon_\phi(v_4,u_4)\setminus\Upsilon_\phi(v_1,u_1)$.  If the former holds, then we proceed as above since we are guaranteed that $\{\alpha,\gamma\} \neq \Upsilon_\phi(v_4,u_4)$, and a similar argument holds in the latter case.  

In Subcase \ref{1.3.1}, we will assume that $\beta,\gamma$ are chosen so that $\{\alpha,\gamma\} \neq \Upsilon_\phi(v_4,u_4)$.  Additionally, as $\sigma(u_2v_2), \sigma(u_2u_3), \sigma(u_3v_3) \notin \{1,2\}$ and are distinct from each other, we may assume that $\sigma(u_3v_3) = 3, \sigma(u_2v_2) = 4$ and $\sigma(u_2u_3) = 5$.

 Since $A_\sigma(u_1u_2)$ and $A_\sigma(u_3u_4)$ are possibly empty, we proceed by considering whether they are empty or not.

\begin{subsubsubcase}\label{1.3.1.1}
$A_\sigma(u_1u_2) = A_\sigma(u_3u_4) = \emptyset$.
\end{subsubsubcase}

As $u_1u_2, u_3u_4$ each see all nine colors and $v_2v_3$ was an edge of $G'$, we may assume that $\Upsilon_\sigma(v_1,u_1) = \Upsilon_\sigma(v_3,u_3) = \{6,7\}$ and $\Upsilon_\sigma(v_2,u_2) = \Upsilon(v_4,u_4) = \{8,9\}$.  Therefore, we can adjust $\sigma$ by uncoloring $u_0u_1, u_4u_5$ and then coloring $u_1u_2$ and $u_3u_4$ with 1 and 2, respectively.  Call this good partial coloring $\psi$.  Since $\Upsilon_\sigma(v_1,u_1) \cap \Upsilon_\sigma(v_4,u_4) = \emptyset$, we can assume that $\beta, \gamma$ were chosen so that  $\{\alpha,\beta\} \neq \{6,7\}$ and $\{\alpha,\gamma\} \neq \{8,9\}$.

Note that $|A_\psi(u_iu_{i+1})| \ge 2$ for $i \in \{0,4,5,8\}$ modulo 9 and $|A_\psi(u_5u_0)| \ge 5$.  In particular, $A_\psi(u_4u_5) \subseteq \{4,6,7\}$ and $A_\psi(u_0u_1) \subseteq \{3,8,9\}$ so that $|A_\psi(u_4u_5) \cup A_\psi(u_0u_1)| \ge 4$.  

Now, suppose $A_\psi(u_4u_5) = A_\psi(u_5u_6)$ and $|A_\psi(u_4u_5)| = 2$.  As $u_4u_5$ sees edges colored 8 and 9, and $\Upsilon_\psi(v_4,u_4) \cap \Upsilon_\psi(v_6,u_6) = \emptyset$, we have $8,9 \in \{\alpha,\beta,\gamma\}$.  However, as $|A_\psi(u_4u_5)| = 2$, $\beta \notin \{8,9\}$ so that $\{8,9\} = \Upsilon_\psi(v_4,u_4) = \{\alpha,\gamma\}$, a contradiction.  Thus, we have $|A_\psi(u_4u_5) \cup A_\psi(u_5u_6)| \ge 3$, and by a symmetric argument, $|A_\psi(u_0u_1) \cup A_\psi(u_8u_0) | \ge 3$.  Thus, we obtain a good coloring of $G$ by SDR.

\begin{subsubsubcase}\label{1.3.1.2}
There exists $\delta \in A_\sigma(u_1u_2)$ and $A_\sigma(u_3u_4) = \emptyset$.
\end{subsubsubcase}

As $u_3u_4$ sees all nine colors, we may assume that $\Upsilon_\sigma(v_3, u_3) = \{6,7\}$ and $\Upsilon_\sigma(v_4,u_4) = \{8,9\}$.  We can adjust $\sigma$ by uncoloring $u_4u_5$ and then coloring $u_3u_4$ with 2 and $u_1u_2$ with $\delta$.  Call this good partial coloring $\psi$.

Observe that $|A_\psi(u_8u_0)| \ge 1$, $|A_\psi(u_4u_5)|, |A_\psi(u_5u_6)| \ge 2$ and $|A_\psi(u_5u_0)|\ge 4$.  If $|A_\psi(u_4u_5)$ $\cup A_\psi(u_5u_6)| \ge 3$, then we obtain a good coloring of $G$ by SDR.  So we have $A_\psi(u_4u_5) = A_\psi(u_5u_6)$ and $|A_\psi(u_4u_5)| = 2$.  However, a similar argument to that used in Subcase \ref{1.3.1.1} implies that $\{\alpha,\gamma\} = \Upsilon_\psi(v_4,u_4)$, a contradiction.

\begin{subsubsubcase}
There exists $\epsilon \in A_\sigma(u_3u_4)$ and $A_\sigma(u_1u_2) = \emptyset$.
\end{subsubsubcase}

Note that the choice of $\beta$ and $\gamma$ does not affect $A_\sigma(u_1u_2)$ or $A_\sigma(u_3u_4)$.  Thus, we can rechoose $\beta$ and $\gamma$, if necessary, so that $\{\alpha, \beta\} \neq \Upsilon_\phi(v_1,u_1)$.  We then repeat a symmetric argument to the above.

\begin{subsubsubcase}
There exist $\delta \in A_\sigma(u_1u_2)$ and $\epsilon \in A_\sigma(u_3u_4)$.
\end{subsubsubcase}

Suppose first that $2 \notin \Upsilon_\sigma(v_3,u_3)$.  We can adjust $\sigma$ by uncoloring $u_4u_5$ and then coloring $u_3u_4$ with 2 and $u_1u_2$ with $\delta$.  From here, the argument is identical to that in Subcase \ref{1.3.1.2}.  Thus, $2 \in \Upsilon_\sigma(v_3,u_3)$.  By symmetry, we also have $1 \in \Upsilon_\sigma(v_2,u_2)$.

We can adjust $\sigma$ by uncoloring $u_2u_3$ and then coloring $u_3u_4, u_1u_2, u_2u_3$ in this order.  As each of these edges sees 1, 2, 3 and 4, we may assume that they are colored 5, 6, 7, respectively.  Call this good partial coloring $\psi$.  Observe that $|A_{\psi}(u_5u_6)|, |A_{\psi}(u_8u_0)| \ge 1$ and $|A_{\psi}(u_5u_0)| \ge 3$.  

If $|A_{\psi}(u_5u_6) \cup A_{\psi}(u_8u_0)| \ge 2$, then we obtain a good coloring of $G$ by SDR.  So we have $A_{\psi}(u_5u_6) = A_{\psi}(u_8u_0) = \{\zeta\}$.  Since $u_5u_6$ sees an edge colored 5, we cannot have $5 \in \{\alpha,\beta,\gamma\}$.  Since $A_\psi(u_8u_0) = A_\psi(u_5u_6)$, $u_8u_0$ also sees 5, and so, $5 \in \Upsilon_\psi(v_8,u_8)$.  Since $v_8v_1$ is an edge of $G'$, we cannot have $5 \in \Upsilon_\psi(v_1,u_1)$.  Similarly, as $|A_\psi(u_8u_0)| = 1$ and $u_8u_0$ sees 1, we cannot have $1 \in \Upsilon_\psi(v_8,u_8)$.

Thus, if we recolor $u_0u_1$ with 5, color $u_8u_0$ with 1, we can than color $u_5u_6$ and $u_5u_0$ in this order to obtain a good coloring of $G$.

This completes the proof of Subcase \ref{1.3.1}.

\begin{subsubcase}
$\phi(u_2v_2) = 1$.
\end{subsubcase}

We can extend $\phi$ to a good partial coloring of $G$, call it $\sigma$, such that $u_4u_5$ is colored with 2, and $u_6u_7$ and  $u_7u_8$ are colored with $\beta$ and $\gamma$, respectively.  Just as in Subcase \ref{1.3.1}, we can choose $\beta, \gamma$ so that $\{\alpha,\beta\} \neq \Upsilon_\sigma(v_1,u_1)$, and additionally require that $\{\alpha,\gamma\} \neq \Upsilon_\sigma(v_4,u_4)$ when $\Upsilon_\sigma(v_1,u_1) \cap \Upsilon_\sigma(v_4,u_4) = \emptyset$.   Also, as $\sigma(u_3v_3) \neq 2$, we may assume that $\sigma(u_3v_3) = 3$.

Note that here, $\sigma$ does not color $u_2u_3$.  Thus, we proceed based on whether or not we can extend $\sigma$ to $u_1u_2, u_2u_3, u_3u_4$.

\begin{subsubsubcase}
We cannot extend $\sigma$ by coloring $u_1u_2, u_2u_3, u_3u_4$.
\end{subsubsubcase}

As $|A_\sigma(u_iu_{i+1})| \ge 2$ for $i \in [3]$, we may assume that $A_\sigma(u_1u_2) = A_\sigma(u_2u_3) = A_\sigma(u_3u_4) = \{4,5\}$.  So without loss of generality, $\Upsilon_\sigma(v_2,u_2) = \Upsilon_\sigma(v_4,u_4) = \{8,9\}$ and $\Upsilon_\sigma(v_1,u_1) = \Upsilon_\sigma(v_3,u_3) = \{6,7\}$.  Recall that just as in Subcase \ref{1.3.1}, $\Upsilon_\sigma(v_1,u_1) \cap \Upsilon_\sigma(v_4,u_4) = \emptyset$, we may assume $\{\alpha,\beta\} \neq \{6,7\}$ and $\{\alpha,\gamma\} \neq \{8,9\}$.  

Now, we can adjust $\sigma$ by uncoloring $u_4u_5$, coloring $u_3u_4$ with 2,  and then coloring $u_1u_2, u_2u_3$ from $\{4,5\}$ so that $u_1u_2$ is not colored with $\beta$.  We call this good partial coloring of $G$, $\psi$, and we may assume that $\psi(u_1u_2) = 4, \psi(u_2u_3) = 5$.

Observe that $|A_\psi(u_iu_{i+1})| \ge 2$ for $i \in \{0,4,5,8\}$ modulo 9 and $|A_\psi(u_5u_0)| \ge 4$.   In particular, $A_\psi(u_4u_5) \subseteq \{4,6,7\}$, $A_\psi(u_0u_1) \subseteq \{3,8,9\}$ and $|A_\psi(u_4u_5) \cup A_\psi(u_0u_1)| \ge 4$.  Now, as $\beta \neq 4$, we have $4 \in A_\psi(u_4u_5)$, and additionally, $4 \notin A_\psi(u_8u_0) \cup A_\psi(u_0u_1)$.  

Also, $|A_\psi(u_8u_0) \cup A_\psi(u_0u_1)| \ge 3$, otherwise we can apply an argument similar to that used in Subcase \ref{1.3.1.1} to show that $\{\alpha,\beta\} = \Upsilon_\psi(v_1,u_1)$, a contradiction.  Thus, we can color $u_4u_5$ with 4, and then obtain a good coloring of $G$ by SDR from the rest.

\begin{subsubsubcase}
We can extend $\sigma$ by coloring $u_1u_2, u_2u_3, u_3u_4$.
\end{subsubsubcase}

Without loss of generality, we may assume that $u_1u_2, u_2u_3, u_3u_4$ are colored with $4,5,6$, respectively, and call this good partial coloring $\psi$.  Observe that $|A_\psi(u_5u_6)| \ge 1$, $|A_\psi(u_8u_0)|$, $|A_\psi(u_0u_1)| \ge 2$ and $|A_\psi(u_5u_0)| \ge 3$.  Additionally, $|A_\psi(u_8u_0) \cup A_\psi(u_0u_1)| \ge 3$, otherwise we can apply an argument similar to that used in Subcase \ref{1.3.1.1} to show that $\{\alpha,\beta\} = \Upsilon_\psi(v_1,u_1)$ (observe that  $|A_\psi(u_0u_1)| = 2$ implies that $|\Upsilon_\psi(v_1,u_1) \cup \{1,2,4,5,\gamma\}| = 7$).  

First,  $\beta,\gamma \notin \{4,6\}$, otherwise $|A_\psi(u_5u_0)| \ge 4$, and we obtain a good coloring of $G$ by SDR.

Additionally, $1 \in \Upsilon_\psi(v_8,u_8)$, otherwise we can color $u_8u_0$ with 1 and then color $u_5u_6$, $u_0u_1$, $u_5u_0$ in this order to obtain a good coloring of $G$.  

We claim $6 \in \Upsilon_\psi(v_1,u_1)$.  If on the contrary, $6 \notin \Upsilon_\psi(v_1,u_1)$, then as $\gamma \neq 6$,  we could color $u_0u_1$ with 6.  Then we have $A_\psi(u_5u_6) = \{\delta\}$ and $A_\psi(u_8u_0) = \{6,\delta\}$, otherwise we could color $u_5u_6, u_8u_0, u_5u_0$ in this order to obtain a good coloring of $G$.  However, since $|A_\psi(u_8u_0) \cup A_\psi(u_0u_1)| \ge 3$ (so that $A_\psi(u_0u_1) \neq \{6,\delta\}$), we can color $u_5u_6$ with $\delta$, $u_8u_0$ with 6 and then color $u_0u_1, u_5u_0$ in this order to obtain a good coloring of $G$.

We may also assume that $\alpha = 6$.  Observe that $6 \notin \{\beta, \gamma\}$, and as $v_8v_1$ is an edge of $G'$, $6 \notin \Upsilon_\psi(v_8,u_8)$.  Thus, if $\alpha \neq 6$, we can color $u_8u_0$ with 6 and then color $u_5u_6, u_0u_1, u_5u_0$ in this order to obtain a good coloring of $G$.

Now, we also have $4 \in \Upsilon_\psi(v_6,u_6)$.  If not, then since $4 \notin \{\beta,\gamma\}$, we can color $u_5u_6$ with 4 and then color $u_0u_1, u_8u_0, u_5u_0$ in this order to obtain a good coloring of $G$.  As $v_4v_6$ is an edge of $G'$, we have $4 \notin \Upsilon_\psi(v_4,u_4)$.

Lastly, we claim that $2 \in \Upsilon_\psi(v_6,u_6)$.  If not, then we can recolor $u_4u_5$ with 4, color $u_5u_6$ with 2 and then color $u_8u_0, u_0u_1, u_5u_0$ in this order to obtain a good coloring of $G$.

Now, we uncolor the edges $u_6u_7, u_7u_8$, and call this new coloring $\tau$.  Observe that $|A_{\tau}(u_iu_{i+1})|$ $\ge 3$ for $i \in \{0,6,7\}$ modulo 9, $|A_{\tau}(u_8u_0)| \ge 4$ and $|A_\tau(u_5u_6)|, |A_{\tau}(u_5u_0)| \ge 5$.  If $|A_{\tau}(u_6u_7) \cup A_{\tau}(u_0u_1)| \ge 6$, then we obtain a good coloring of $G$ by SDR.  Thus, there exists some $\epsilon$ such that we can color $u_6u_7, u_0u_1$ with $\epsilon$ and then color $u_7u_8, u_5u_6, u_8u_0, u_5u_0$ in this order to obtain a good coloring of $G$.

This completes all subcases of Case \ref{Case1}.

\begin{case}\label{different}
$\phi(u_1v_1) = \phi(u_4v_4)$.
\end{case}

Without loss of generality, we may assume that $\phi(u_iv_i) = 1$ for $i \in \{1,4,6,8\}$,  $\phi(u_{2}v_{2}) = 2$ and $\phi(u_{3}v_{3}) = 3$.

\begin{subcase}\label{2.1}
We can extend $\phi$ by coloring $u_{1}u_{2}, u_{2}u_{3}, u_{3}u_{4}$.
\end{subcase}

Let us extend $\phi$ by coloring $u_1u_2, u_2u_3, u_3u_4$, and then uncolor $u_7v_7$.  Call this new good partial coloring $\sigma$.  Without loss of generality, we may assume that $\sigma(u_1u_2) = 4, \sigma(u_2u_3) = 5, \sigma(u_3u_4) = 6$.

\begin{subsubcase}\label{TT}\label{CASE2-1-1}
Either $6 \notin \Upsilon_\sigma(v_1,u_1)$ or $4 \notin \Upsilon_\sigma(v_{4},u_{4})$.
\end{subsubcase}
 
By symmetry, we may assume that $4 \notin \Upsilon_\sigma(v_4,u_4)$.  As a result, we can extend $\sigma$ by coloring $u_4u_5$ with 4.  Call this good partial coloring  $\psi$.  Note that $|A_\psi(u_7v_7)| \ge 2$, $|A_\psi(u_6u_7)|, |A_\psi(u_0u_1)| \ge 3$, $|A_\psi(u_5u_6)|, |A_\psi(u_7u_8)| \ge 4$, $|A_\psi(u_0u_1)| \ge 5$ and $|A_\psi(u_5u_0)| \ge 6$.  

First, we show that $|A_\psi(u_7v_7) \cup A_\psi(u_0u_1)| \ge 5$.  If not, then we can color $u_7v_7, u_0u_1$ with some $\beta$ and then color $u_6u_7, u_5u_6, u_7u_8, u_8u_0, u_5u_0$ in this order to obtain a good coloring of $G$.  In a similar manner, we show that $|A_\psi(u_6u_7) \cup A_\psi(u_0u_1)| \ge 6$ by otherwise coloring $u_6u_7, u_0u_1$ with some $\gamma$, and then coloring $u_7v_7, u_7u_8, u_5u_6, u_8u_0, u_5u_0$ in this order to obtain our good coloring of $G$.

Now, if $|A_\psi(u_7v_7) \cup A_\psi(u_5u_0)| \ge 7$, then we can obtain a good coloring of $G$ by SDR.  Otherwise, we can color $u_7v_7, u_5u_0$ with some $\delta$, and then obtain a good coloring of $G$ by SDR from the remaining edges using the above.

\begin{subsubcase}\label{CASE2-1-2}
$6 \in \Upsilon_\sigma(u_{1}v_{1})$ and $4 \in \Upsilon_\sigma(u_{4}v_{4})$.
\end{subsubcase}

We first note that there exists $\beta \in A_\sigma(u_7v_7) \setminus\{4\}$ and that $4 \in A_\sigma(u_5u_6)$.  Thus, we can obtain another good partial coloring of $G$ by coloring $u_5u_6$ with 4 and $u_7v_7$ with $\beta$.  Call this new coloring $\psi$.  Observe $|A_\psi(u_6u_7)|, |A_\psi(u_7u_8)| \ge 2$, $|A_\psi(u_4u_5)|, |A_\psi(u_0u_1)| \ge 3$, $|A_\psi(u_8u_0)| \ge 4$ and $|A_\psi(u_5u_0)| \ge 6$.  

First, if $|A_\psi(u_6u_7) \cup A_\psi(u_0u_1)| \ge 5$, then we obtain a good coloring of $G$ by SDR.  Thus, there exists some $\gamma \in A_\psi(u_6u_7) \cap A_\psi(u_0u_1)$ so that we can color $u_6u_7, u_0u_1$ with $\gamma$ and then color $u_7u_8, u_8u_0, u_4u_5, u_5u_0$ in this order to obtain a good coloring of $G$.

\begin{subcase}
We cannot extend $\phi$ by coloring  $u_{1}u_{2}, u_{2}u_{3}, u_{3}u_{4}$.
\end{subcase}

As $|A_\phi(u_iu_{i+1})| \ge 2$ for $i \in \{1,2,3\}$, we may assume that $A_\phi(u_iu_{i+1}) = \{8,9\}$ for $i \in \{1,2,3\}$.  Thus, without loss of generality,  $\Upsilon_\phi(v_{2},u_{2}) = \Upsilon_\phi(v_{4},u_{4}) = \{4, 5\}$ and $\Upsilon_\phi(v_{1},u_{1}) = \Upsilon_\phi(v_{3},u_{3}) = \{6, 7\}$.  We can recolor $u_4v_4$ with some $\beta \neq 1$ and still maintain a good partial coloring of $G$.

Thus, we can obtain another good partial coloring of $G$ by first recoloring $u_4v_4$ with $\beta$, color $u_3u_4$ with 1 and then color $u_2u_3, u_1u_2$ in this order.  As in Subcase \ref{2.1}, we also uncolor $u_7v_7$, and call this new coloring $\sigma$.  Note that $\{\sigma(u_1u_2), \sigma(u_2u_3)\} = \{8,9\}$, and so without loss of generaltiy, $\sigma(u_1u_2) = 8, \sigma(u_2u_3) = 9$.

\begin{subsubcase}
$\beta \neq 8$.
\end{subsubcase}

As $8 \in A_\phi(u_3u_4)$, we cannot have $8 \in \sU_\sigma(y_4)$.  Thus, we can extend $\sigma$ by coloring $u_4u_5$ with 8 and then proceed in the same way as in Subcase \ref{CASE2-1-1} replacing 8 with 4.

\begin{subsubcase}
$\beta = 8$.
\end{subsubcase}

By the existence of $v_8v_1$ in our auxiliary graph $G$, $6 \in \Upsilon_\sigma(v_1,u_1)$ implies that $6 \notin \Upsilon_\sigma(v_1,u_1)$ so that $6 \in A_\sigma(u_8u_0)$.  Note that there exists some $\gamma \in A_\sigma(u_7v_7) \setminus\{6\}$.  

We can then extend $\sigma$ to another good coloring of $G$ by coloring $u_7v_7$ with $\gamma$ and $u_8u_0$ with 6.  Call this $\psi$.  Observe that $A_\psi(u_4u_5) = \{2,7\}, A_\psi(u_0u_1) = \{3,4,5\}$, $|A_\psi(u_6u_7)|$, $|A_\psi(u_7u_8)|$ $\ge 2$, $|A_\psi(u_5u_6)| \ge 3$ and $|A_\psi(u_5u_0)| \ge 6$.  As $A_\psi(u_4u_5) \cap A_\psi(u_0u_1) = \emptyset$, coloring $u_4u_5$ does not affect coloring $u_0u_1$.  

Now, if $|A_\psi(u_4u_5) \cup A_\psi(u_7u_8)| \ge 4$, we can color $u_4u_5, u_5u_6, u_6u_7, u_7u_8$ by SDR and then color $u_0u_1, u_5u_0$ in this order to obtain a good coloring of $G$.  Thus, there exists some $\delta$ so that we can color $u_4u_5, u_7u_8$ with $\delta$ and then color $u_6u_7, u_5u_6, u_0u_1, u_5u_0$ in this order to obtain a good coloring of $G$.

This completes the proof of the final subcase of Case \ref{different}, and so proves the lemma.
\end{proof}

%%%%%%%%%%%%%%%%%%%%%%%%%%%%%%%%%%%%%%%%

\section{Proof of Theorem \ref{thm:conj}}\label{sec:proof}

We are now ready to prove Theorem \ref{thm:conj} via discharging using the lemmas from Sections \ref{struct1}, \ref{struct2} and \ref{struct3},

\begin{proof}
By Euler's formula,
\[
\sum_{v\in V(G)}(2d(v)-6)+\sum_{f\in F(G)}(d(f)-6)=-12.
\]

Thus, if we assign  to each vertex $v$ the initial charge $2d(v) - 6$ and  to each face $f$ the initial charge $d(f)-6$, 
then the total charge will be $ - 12$. We design appropriate discharging rules and redistribute charges among faces and vertices so
 that the final charge of every face and every vertex is nonnegative, a contradiction.

{\bf Discharging Rules:}\\
(R1) Every 2-vertex receives 1 from each incident face.\\
(R2) Every 5-face receives $\frac{1}{5}$ from each adjacent face. \\

By %Lemmas \ref{NoTriangle}, \ref{No4cycle}, \ref{No2on5cycle}, \ref{No2on6cycle} and \ref{No2on7face}, 
 Rule~(R1), at the end of discharging, each $2$-vertex %is incident with two $8^{+}$-faces.  As only faces give charge, each 2-vertex 
 will have  charge $-2+1+1=0$. The charge of each 3-vertex does not change and remains $0$.

%By Lemmas \ref{No5-5face} and \ref{No5-6face}, every $5$-face is adjacent to five $7^{+}$-faces, and so,
 By Rule~(R2) and Lemmas \ref{No2on5cycle} and \ref{No5-5face}, the final charge of every $5$-face is $5 - 6 + 5 \times \frac{1}{5} = 0$.

By Lemmas \ref{No2on6cycle} and \ref{No5-6face}, each 6-face gives no charge.  Thus, as it starts with zero charge and does not receives any charge, the final charge is zero.

By Lemmas \ref{No2on7face} and \ref{No5-5face}, each 7-face contains only 3-vertices and is adjacent to at most three 5-faces.  Thus, the final charge is at least $7 - 6 - 3 \times \frac{1}{5} = \frac{2}{5}$.

By Lemmas \ref{No5-5face} and \ref{distance>=5}, each $k$-face, $k \ge 8$, is adjacent to at most $\lfloor \frac{k}{2} \rfloor$ 5-faces 
and contains at most $\lfloor \frac{k}{5}\rfloor$ 2-vertices on its boundary.  
Thus, the final charge is at least  
$k - 6 - \left \lfloor \frac{k}{5} \right \rfloor \times 1 - \left \lfloor \frac{k}{2} \right \rfloor \times \frac{1}{5}$, which is positive for $k \ge 8$.

This completes the proof.
\end{proof}

%%%%%%%%%%%%%%%%%%%%%%%%%%%%%%%%%%%%%%%%%%%%%%%%%%%%%%%%%%

\noindent\textbf{Conclusion.}  There are many unresolved questions regarding the strong chromatic index of graphs.  We present a few that pertain specifically to subcubic planar graphs.  As mentioned, Theorem \ref{thm:conj} is shown to be best possible by the complement of $C_6$.  To the authors' knowledge, this is the only such example. Perhaps the result can be improved for graphs outside of a potentially finite family.  Additionally, a list-coloring result is unknown and does not extend naturally from the proofs given in this paper.  Thus, a list-coloring result similar to that of Theorem \ref{thm:conj} would be of interest.\\

%%%%%%%%%%%%%%%%%%%%%%%%%%%%%%%%%%%%%%%

\noindent {\bf Acknowledgment.}  The authors thank the referees for their comments and careful reading of this paper.

%%%%%%%%%%%%%%%%%%%%%%%%%%%%%%%%%%%%%


\begin{thebibliography}{10}



\bibitem{Andersen}  L.D. Andersen, The strong chromatic index of a cubic graph is at most 10, Discrete Math. 108 (1992) 231--252.

%\bibitem{BIKMTT}  C.L. Barrett, G. Istrate, V.S.A. Kumar, M.V. Marathe, S. Thite and S. Thulasidasan, Strong edge coloring for channel assignment in wireless radio networks, In proceedings of the 4th annual IEEE international conference on Pervasive Computing and Communications Workshops, 2006, 106--110.

\bibitem{BorodinIvanova}  O.V. Borodin and A.O. Ivanova, Precise upper bound for the strong edge chromatic number of sparse planar graphs, Discuss. Math. Graph Theory 33 (2013), no. 4, 759--770.

\bibitem{BJ} H. Bruhn and F. Joos, A stronger bound for the strong chromatic index, arXiv:1504.02583 [math.CO].

%\bibitem{CMPR}  G.J. Chang, M. Montassier, A. P\^{e}cher and A. Raspaud, Strong chromatic index of planar graphs with large girth, Discuss. Math. Graph Theory (to appear).

%\bibitem{Cranston}  D. Cranston, Strong edge-coloring graphs with maximum degree 4 using 22 colors, Discrete Math. 306 (2006) 2772--2778.

\bibitem{Erdos}  P. Erd\H{o}s, Problems and results in combinatorial analysis and graph theory, Discrete Math. 72 (1988) 81--92.

\bibitem{ErdosNesetril}  P. Erd\H{o}s and J. Ne\v{s}et\v{r}il, [Problem], in: G. Hal\'{a}sz and V.T. S\'{o}s (eds.), Irregularities of Partitions, Springer, Berlin, 1989, 161--165.

\bibitem{FGST}  R.J. Faudree, R.H. Schelp, A. Gy\'{a}rf\'{a}s and Zs. Tuza, The strong chromatic index of graphs, Ars Combin. 29B (1990) 205--211.

\bibitem{FJ1}  J.L. Fouquet and J. Jolivet, Strong edge-coloring of cubic planar graphs, Progress in Graph Theory (Waterloo 1982), 1984, 247--264.

\bibitem{FJ2}  J.L. Fouquet and J. Jolivet, Strong edge-coloring of graphs and applications to multi-$k$-gons, Ars Combin. 16A (1983) 141--150.

\bibitem{Hall} P. Hall, On representatives of subsetes, J. London Math. Soc. 10 (1935), 26--30.

\bibitem{HR}  M. A. Henning and D. Rautenbach, Induced matchings in subcubic graphs without short cycles, Discrete Math. 315--316 (2014) 165--172.

\bibitem{HMRV}  H. Hocquard, M. Montassier, A. Raspaud and P. Valicov, On strong edge-colouring of subcubic graphs, Discrete Appl. Math. 161 (2013) 2467--2479.

%\bibitem{HOV} H. Hocquard, P. Ochem and P. Valicov, Strong edge-colouring and induced matchings, Inform. Process. Lett. 113 (2013) 836--843.

\bibitem{HV}  H. Hocquard and P. Valicov, Strong edge colouring of subcubic graphs, Discrete Appl. Math. 159 (2011) 1650--1657.

\bibitem{HQT}  P. Hor\'{a}k, H. Qing and W.T. Trotter, Induced matchings in cubic graphs, J. Graph Theory 17 (1993) 151--160.

\bibitem{HLSS}  D. Hud\'{a}k, B. Lu\v{z}ar, R. Sot\'{a}k, R. \v{S}krekovski, Strong edge-coloring of planar graphs, Discrete Math. 324 (2014) 41--49.

\bibitem{JRS}  F. Joos, D. Rautenbach and T. Sasse, Induced matchings in subcubic graphs, SIAM J. Discrete Math. 28 (2014), no. 1, 468--473.

\bibitem{KMM}  R.J. Kang, M. Mnich and T. M\"{u}ller, Induced matchings in subcubic planar graphs, SIAM J. Discrete Math. 26 (2012), no. 3, 1383--1411.

\bibitem{MolloyReed}  M. Molloy and B. Reed, A bound on the strong chromatic index of a graph, J. Combin. Theory, Series B 69 (1997) 103--109.

%\bibitem{N}  K. Nakprasit, A note on the strong chromatic index of bipartite graphs, Discrete Math. 308 (2008) 3726--3728.

\bibitem{NN}  K. Nakprasit and K. Nakprasit, The strong chromatic index of graphs and subdivisions, Discrete Math. 317 (2014) 75--78.

\bibitem{NKGB}  T. Nandagopal, T. Kim, X. Gao and V. Bharghavan, Achieving MAC layer fairness in wireless packet networks In Proc. 6th ACM Conf. on Mobile Computing and Networking, 2000, 87--98.

\bibitem{R}  S. Ramanathan, A unified framework and algorithm for (T/F/C) DMA channel assignment in wireless networks, In Proc. IEEE INFOCOM '97, 900--907, 1997.

%\bibitem{RL}  S. Ramanathan and E.L. Lloyd. Scheduling algorithms for multi-hop radio networks.  In IEEE/ACM Trans. Networking, vol. 2, 166--177, 1993.

\bibitem{StegerYu}  A. Steger and M.L. Yu, On induced matchings, Discrete Math. 120 (1993) 291--295.


\bibitem{WuLin}  J. Wu and W. Lin, The strong chromatic index of a class of graphs, Discrete Math. 308 (2008) 6254--6261.




\end{thebibliography}
\end{document}